\documentclass[11pt,a4paper,reqno]{amsart}
\usepackage{amsfonts}
\usepackage{amsmath,amssymb,amsthm,amsxtra}
\usepackage{float}
\usepackage[colorlinks,
  linkcolor=blue, anchorcolor=Periwinkle,
  citecolor=Orange, urlcolor=black]{hyperref}
\usepackage[usenames,dvipsnames]{xcolor}
\usepackage{enumitem}
\usepackage{graphicx}
\usepackage{subfigure}
\usepackage{bookmark}
\usepackage{tikz}
\usetikzlibrary{matrix}
\usetikzlibrary{decorations.pathmorphing}
\usepackage{tikz-cd}
\usepackage{url}
\usepackage[all]{xy}

\author{Yu Qiu and Jon Woolf}
\title{Contractible stability spaces and faithful braid group actions}
\newtheorem{theoremalpha}{Theorem}

\newtheorem{theorem}{Theorem}[section]
\newtheorem*{theorem*}{Theorem}
\newtheorem{proposition}[theorem]{Proposition}
\newtheorem{corollary}[theorem]{Corollary}
\newtheorem{lemma}[theorem]{Lemma}

\theoremstyle{definition}

\newtheorem{definition}[theorem]{Definition}
\newtheorem{example}[theorem]{Example}

\newtheorem{remark}[theorem]{Remark}

\newtheorem{construction}[theorem]{Construction}

\newcommand{\defn}[1]{\emph{#1}}

\newcommand{\ie}{i.e.\ }

\newcommand{\cf}{cf.\ }

\newcommand{\N}{\mathbb{N}}
\newcommand{\Z}{\mathbb{Z}}

\newcommand{\R}{\mathbb{R}}
\newcommand{\C}{\mathbb{C}}
\newcommand{\U}{\mathbb{H}}
\renewcommand{\P}{\mathbb{P}}

\newcommand{\h}{\mathfrak{h}}
\newcommand{\g}{\mathfrak{g}}
\newcommand{\hreg}{\h^{\text{reg}}}

\newcommand{\Aut}{\operatorname{Aut}}

\DeclareMathOperator{\KGdim}{KGdim}
\DeclareMathOperator{\im}{Im}
\DeclareMathOperator{\re}{Re}
\DeclareMathOperator{\rk}{rank}

\newcommand{\cat}[1]{\mathcal{#1}}
\newcommand{\mor}[2]{\Mor{}{#1}{#2}}
\newcommand{\Mor}[3]{\mathrm{Hom}_{#1}\!\left(#2,#3\right)}
\newcommand{\GrMor}[4]{\mathrm{Hom}_{#1}^{#2}\!\left(#3,#4\right)}
\newcommand{\Cone}{\operatorname{Cone}}
\newcommand{\ab}{\mathrm{Ab}}
\newcommand{\epic}{\twoheadrightarrow}

\newcommand{\ts}[1]{\cat{#1}}
\newcommand{\aisle}[2]{\ts{#1}^{\leq #2}}
\newcommand{\coaisle}[2]{\ts{#1}^{\geq #2}}
\newcommand{\heart}[1]{\ts{#1}^{0}}
\newcommand{\lt}[2]{L_{#2}{\ts{#1}}}
\newcommand{\rt}[2]{R_{#2}{\ts{#1}}}

\newcommand{\slicing}{\cat{P}}
\newcommand{\charge}{Z}
\newcommand{\Stab}{\operatorname{Stab}}

\newcommand{\stab}[1]{\Stab\!\left(\cat{#1}\right)}
\newcommand{\stabo}[1]{\Stab^\circ\!\left(\cat{#1}\right)}
\newcommand{\stabalg}[1]{\Stab_\mathrm{alg}\!\left(\cat{#1}\right)}
\newcommand{\stabalgo}[1]{\Stab^\circ_\mathrm{alg}\!\left(\cat{#1}\right)}
\newcommand{\stabalgt}[1]{\Stab^\theta_\mathrm{alg}\!\left(\cat{#1}\right)}

\newcommand{\PT}[1]{\mathrm{T}(\cat{#1})} 
\newcommand{\tilt}[1]{\mathrm{Tilt}(\cat{#1})} 
\newcommand{\tilto}[1]{\mathrm{Tilt}^\circ(\cat{#1})} 

\newcommand{\tiltalg}[1]{\mathrm{Tilt}_\mathrm{alg}(\cat{#1})} 
\newcommand{\tiltalgo}[1]{\mathrm{Tilt}^\circ_\mathrm{alg}(\cat{#1})} 
\newcommand{\PI}[1]{\mathrm{Int}(\cat{#1})} 

\newcommand{\tsleq}{\subseteq} 
\newcommand{\tileq}{\leq} 
\newcommand{\atileq}{\preccurlyeq} 

\newcommand{\its}[2]{\left[ {#1}, {#2}\right]_\tsleq} 
\newcommand{\itilt}[2]{\left[  {#1}, {#2}\right]_\tileq} 
\newcommand{\ialg}[2]{\left[  {#1}, {#2}\right]_\atileq} 

\newcommand{\CoveringPoset}{\mathcal{P}}

\newcommand{\Br}{\operatorname{Br}}
\newcommand{\braidQ}{\Br\left(Q\right)}

\newcommand{\GQ}[1]{\Gamma_{\!#1}Q}
\newcommand{\catGQ}[1]{\cat{D}\left( \Gamma_{\!#1}Q\right)}

\newcommand{\tiltGQ}[1]{\mathrm{Tilt}\!\left( \Gamma_{\!#1}Q\right)}
\newcommand{\tiltGQo}[1]{\mathcal{I}_{\Gamma_{\!#1}Q}}

\newcommand{\tiltQ}{\mathrm{Tilt}^\circ\!\left( Q\right)}
\newcommand{\tiltQo}[1]{\mathcal{I}_{Q}}

\newcommand{\stabGQ}[1]{\Stab\!\left( \Gamma_{\!#1}Q\right)}

\newcommand{\braidGQ}[1]{\operatorname{Br}\left( \Gamma_{\!#1}Q\right)}

\newcommand{\braidsurj}{\Phi} 

\newcommand{\GQzero}{\ts{D}_{\Gamma Q}} 
\newcommand{\Qzero}{\ts{D}_{Q}} 

\newcommand{\LI}[1]{\mathcal{L}^{#1}} 

\newcommand{\cluster}[2]{\cat{C}_{#1}\left(#2\right)} 
\newcommand{\mutation}[2]{\cat{CM}_{#1}\left(#2\right)} 
\newcommand{\shift}[1]{\operatorname{\Sigma}_{#1}}   

\def\Ext{\operatorname{Ext}}
\def\Irr{\operatorname{Irr}}
\def\Add{\operatorname{Add}}
\def\Ind{\operatorname{Ind}}
\begin{document}
\begin{abstract}
We prove that any `finite-type' component of a stability space of a triangulated category is contractible. The motivating example of such a component is the stability space of the Calabi--Yau-$N$ category $\catGQ{N}$ associated to an ADE Dynkin quiver. In addition to showing that this is contractible we prove that the braid group $\braidQ$ acts freely upon it by spherical twists, in particular that the spherical twist group $\braidGQ{N}$ is isomorphic to $\braidQ$. This generalises Brav--Thomas' result for the $N=2$ case. Other classes of triangulated categories with finite-type components in their stability spaces include locally-finite triangulated categories with finite rank Grothendieck group and discrete derived categories of finite global dimension.

    \vskip .3cm
    {\parindent =0pt
    \it Key words:} Stability conditions, Calabi--Yau categories, spherical twists, braid groups
\end{abstract}
\maketitle
\tableofcontents\addtocontents{toc}{\setcounter{tocdepth}{1}}

\section{Introduction}
\label{introduction}
\subsection{Stability conditions}
Spaces of stability conditions on a triangulated category were introduced in \cite{MR2373143}, inspired by the work of Michael Douglas on stability of D-branes in string theory. The construction associates a  space $\stab{C}$ of stability conditions to each triangulated category $\cat{C}$. A stability condition $\sigma \in \stab{C}$ consists of a \defn{slicing} --- for each $\varphi\in \R$ an abelian subcategory  $\slicing_\sigma(\varphi)$ of \defn{semistable objects of phase $\varphi$}  such that each object of $\cat{C}$ can be expressed as an iterated extension of semistable objects --- and a \defn{central charge} $\charge \colon K\cat{C} \to \C$ mapping the Grothendieck group $K\cat{C}$ linearly to $\C$. The slicing and charge obey a short list of axioms.
The miracle is that the space $\stab{C}$ of stability conditions is a complex manifold, locally modelled on a linear subspace of $\mor{K\cat{C}}{\C}$ \cite[Theorem 1.2]{MR2373143}. It carries commuting actions of $\C$, acting by rotating phases and rescaling masses, and of the automorphism group $\Aut(\cat{C})$.

Whilst a number of examples of spaces of stability conditions are known, it is in general difficult to compute $\stab{C}$. It is widely believed that spaces of stability conditions are contractible, and this has been verified in certain examples. We give the first proof of contractibility for certain general classes of triangulated categories satisfying (strong) finiteness conditions.

Our strategy is to identify general conditions under which there are no `complicated' stability conditions. One measure of the complexity of a stability condition $\sigma$ is the phase distribution, \ie the set $\{ \varphi \in \R \mid \slicing_\sigma(\varphi) \neq 0 \}$ of phases for which there is a non-zero semistable object. A good heuristic is that a stability condition with a dense phase distribution is complicated, whereas one with a discrete phase distribution is much less so --- see \cite{MR3289326} for a precise illustration of this principle.

Another measure of complexity is provided by the properties of the heart of the stability condition $\sigma$. This is  the full extension-closed subcategory $\slicing_\sigma(0,1]$ generated by the semistable objects with phases in the interval $(0,1]$. It is the heart of a bounded t-structure on $\cat{C}$ and so in particular is an abelian category. From this perspective the `simplest' stability conditions are those whose heart is Artinian and Noetherian with finitely many isomorphism class of simple objects; we call these \defn{algebraic} stability conditions.

These two measures of complexity are related: if there is at least one algebraic stability condition then the union $\C \cdot \stabalg{\cat{C}}$ of orbits of algebraic stability conditions under the $\C$-action is the set of stability conditions whose phase distribution is not dense.

We show that the subset $\stabalg{\cat{C}}$ is stratified by real submanifolds, each consisting of stability conditions for which the heart is fixed and a given subset of its simple objects have integral phases. Each of these strata is contractible, so the topology of $\stabalg{\cat{C}}$ is governed by the combinatorics of adjacencies of strata. It is well-known that as one moves in $\stab{C}$ the associated heart  changes by Happel--Reiten--Smal\o\ tilts. The combinatorics of tilting is encoded in the poset $\tilt{\cat{C}}$ of t-structures on $\cat{C}$ with relation $\ts{D} \leq \ts{E} \iff$ there is a finite sequence of (left) tilts from $\ts{D}$ to $\ts{E}$. Components of this poset are in bijection with components of $\stabalg{\cat{C}}$. Corollary~\ref{poset-iso} describes the precise relationship between $\tilt{\cat{C}}$ and the stratification of $\stabalg{\cat{C}}$. Using this connection we obtain our main theorem:
\begin{theoremalpha}[Lemma~\ref{finite-type cpts correspondence} and Theorem~\ref{stab contractible}]
Suppose each algebraic t-structure in some component of $\tilt{C}$ has only finitely many tilts, all of which are algebraic. Then the corresponding component of $\stabalg{C}$ is actually a component of $\stab{C}$, and moreover is contractible.
\end{theoremalpha}
We say that a component satisfying the conditions of the theorem has {\em finite-type}. The phase distribution of any stability condition in a finite-type component is discrete. It seems plausible that the converse is true, \ie that any component of $\stab{\cat{C}}$ consisting entirely of stability conditions with discrete phase distribution is a a finite-type component, but we have not been able to prove this. There are several interesting classes of examples of finite-type components. We show that if $\cat{C}$ is
\begin{itemize}
\item a locally-finite triangulated category with finite rank Grothendieck group (\cite{MR2955969}, see Section~\ref{locally-finite categories}), then any component of $\stab{C}$ is of finite-type;
\item a discrete derived category of finite global dimension (see Section~\ref{sec:DDC}), then $\stab{C}$ consists of a single finite-type component;
\item  the bounded derived category $\catGQ{N}$ of finite-dimensional representations of the Calabi--Yau-$N$ Ginzburg algebra of a Dynkin quiver $Q$, for any $N\geq 2$, then the space of stability conditions has finite-type.
\end{itemize}
The bounded derived category $\cat{D}(Q)$ of a Dynkin quiver $Q$ is both locally-finite and discrete, and the first two classes can be seen as different ways to generalise from these basic examples. Perhaps surprisingly, until now the space of stability conditions on $\cat{D}(Q)$ was only known to be contractible for $Q$ of type $A_1$ or $A_2$, although it was known by \cite{MR3281136} that it was simply-connected.

Similarly, for discrete derived categories contractibility was known before only for the simplest case, which was treated in \cite{MR2739061}. The description of the stratification of $\stab{D}$ for $\cat{D}$ a discrete derived category, from which contractibility follows, was obtained independently, and simultaneously with our results, in \cite{pauk2}. They use an alternative algebraic interpretation of the combinatorics of the stratification in terms of silting subcategories and silting mutation.

The third class of examples has been the most intensively studied. The space of stability conditions $\stabGQ{N}$ has been identified as a complex space in various cases, in each of which it is known to be contractible. The connectedness of $\stabGQ{N}$ is proven by \cite{AMY} recently for the Dynkin case.
For $N=2$ and $Q$ a quiver of type $A$ it was first studied in \cite{MR2230573}, where the stability space was shown to be the universal cover of a configuration space of points in the complex plane. Using different methods  \cite{MR2549952} identified  $\stabGQ{2}$ for any Dynkin quiver $Q$ as a covering space using a geometric description in terms of Kleinian singularities. Later \cite{MR2854121}, see also \cite{MR3281136}, showed that it was the {\em universal} cover in all these cases. When the underlying Dynkin diagram of $Q$ is $A_n$, \cite{Akishi} shows that  $\stabGQ{N}$ is the universal cover of the space of degree $n+1$ polynomials $p_n(z)$ with simple zeros.  The central charges are constructed as periods of the quadratic differential $p_n(z)^{N-2}dz^{\otimes 2}$ on $\P^1$, using the technique of \cite{MR3349833}. The $N=3$ case of this result was obtained previously in \cite{sutherland}. The $A_2$ case for arbitrary $N$, including $N=\infty$ which corresponds to stability conditions on $\cat{D}(A_2)$, was treated in \cite{ttq} using different methods. Besides, \cite{ishii2010} showed that $\stabGQ{2}$ is connected, and also that the stability space of the affine counterpart is connected and simply-connected. Our methods do not apply to this latter case. Finally, \cite{MR3498923} proved the contractibility of the principal component of
$\stabGQ{3}$ for any affine $A$ type quivers.

Although there are  several interesting classes of examples, the finiteness condition required for our theorem is strong. For instance it is not satisfied by tame representation type quivers such as the Kronecker quiver. Different methods will probably be required in these cases, because the stratification of the space of algebraic stability conditions fails to be locally-finite and closure-finite, and so is much harder to understand and utilise. Examples of alternative methods for proving the contractibility of the space of stability conditions on $\cat{D}(Q)$ can be found in  \cite{MR2219846} for the case of the Kronecker quiver, and  \cite{dk} for the case of the acyclic triangular quiver.

\subsection{Representations of braid groups}
\label{braid group reps}

One can associate a braid group $\braidQ$ to an acyclic quiver $Q$ --- it is defined by having a generator for each vertex, with a braid relation $aba=bab$ between generators whenever the corresponding vertices are connected by an arrow, and a commuting relation $ab=ba$ whenever they are not. For example, the braid group of the $A_n$ quiver is the standard braid group on $n+1$ strands.

This braid group acts on $\catGQ{N}$ by spherical twists. The image of $\braidQ$ in the group of automorphisms is the Seidel--Thomas braid group $\braidGQ{N}$. Its properties are closely connected to the topology of $\stabGQ{N}$, in particular $\stabGQ{N}$ is simply-connected whenever the Seidel--Thomas braid action on it is faithful.

The Seidel--Thomas braid group originated in the study of Kontsevich's homological mirror symmetry. On the symplectic side, Khovanov--Seidel \cite{MR1862802} showed that when $Q$ has type $A$ the category $\catGQ{N}$ can be realised as a subcategory of the derived Fukaya category of the Milnor fibre of a simple singularity of type $A$. Here $\braidQ$ acts as (higher) Dehn twists along Lagrangian spheres, and they proved this actions is faithful. On the algebraic geometry side, Seidel--Thomas \cite{MR1831820} studied the mirror counterpart of \cite{MR1862802}; here $\catGQ{N}$ can be realised as a subcategory of the bounded derived category of coherent sheaves of the mirror variety.

The proofs of faithfulness of the braid group action by Khovanov--Seidel--Thomas (\cite{MR1862802,MR1831820}) depend on the existence of a faithful geometric representation of the braid group in the mapping class group of a surface. Such faithful actions are known to exist by Birman--Hilden \cite{MR0325959} when $Q$ has type $A$, and by Perron--Vannier \cite{MR1411346} when $Q$ has type $D$. Surprisingly, Wajnryb \cite{MR1719815} showed that there is no such faithful geometric representation of the braid group of type $E$, so this method of proof cannot be generalised to all Dynkin quivers. A different approach, relying on the \emph{Garside structure} on the braid group $\braidQ$, was used by Brav--H.Thomas \cite{MR2854121} to prove that the braid group action on $\catGQ{2}$ is faithful for all Dynkin quivers $Q$. The $N=2$ case is the simplest because $\braidQ$ acts transitively on the tilting poset $\tiltGQ{N}$; this is not so for $N\geq3$. Nevertheless, we are able to `bootstrap' from the $N=2$ case to prove:
\begin{theoremalpha}[Corollaries~\ref{CYN contractibility},~\ref{braid action free 1}, and~\ref{braid action free 2}]
For any Dynkin quiver $Q$ and any $N\geq 2$ the action of $\braidQ$ on $\catGQ{N}$ is faithful, and the induced action on $\stabGQ{N}$ is free. Moreover, $\stabGQ{N}$ is contractible and the finite-dimensional complex manifold  $\stabGQ{N}/\braidQ$ is a model for the classifying space of $\braidQ$.
\end{theoremalpha}

\subsection*{Acknowledgments}
We would like to thank Alastair King for interesting and helpful discussion of this material. Nathan Broomhead, David Pauksztello, and David Ploog were kind enough to share an early version of their preprint \cite{pauk2}. They were also very helpful in explaining the translation between their approach via silting subcategories and the one in this paper via algebraic t-structures.
The second author would also like to thank, sadly posthumously, Michael Butler for his interest in this work, and his guidance on matters algebraic. He is much missed.

\section{Preliminaries}
\label{preliminaries}
Throughout the paper, ${k}$ is a fixed (not necessarily algebraically-closed) field.
The Grothendieck group of an abelian, or triangulated, category $\cat{C}$ is denoted by $K\cat{C}$.

 The bounded derived category of the path algebra $kQ$ of a quiver $Q$ is denoted $\cat{D}(Q)$ and the bounded derived category  of finite-dimensional representations of the Calabi--Yau-$N$ Ginzburg algebra of a Dynkin quiver $Q$, for $N\geq 2$, is denoted $\catGQ{N}$. The bounded derived category of coherent sheaves on a variety $X$ over $k$ is denoted $\cat{D}(X)$. The spaces of locally-finite stability conditions on these triangulated categories are denoted by $\Stab(Q)$, by $\Stab(\Gamma_NQ)$ and by $\Stab(X)$ respectively.

\subsection{Posets}
\label{posets}
Let $P$ be a poset. We denote the closed interval
\[
\{ r\in P \ \colon\  p\leq r\leq q\}
\]
 by $[p,q]$, and similarly use the notation $(-\infty,p]$ and $[p,\infty)$ for bounded above and below intervals. A poset is \defn{bounded} if it has both a minimal and a maximal element. A \defn{chain}  of length $k$ in a poset $P$ is a sequence $p_0 < \cdots < p_k$ of elements. One says $q$ \defn{covers} $p$ if $p<q$ and there does not exist $r\in P$ with $p < r < q$. A chain $p_0< \cdots < p_k$ is said to be \defn{unrefinable} if $p_i$ covers $p_{i-1}$ for each $i=1,\ldots,k$. A \defn{maximal} chain is an unrefinable chain in which $p_0$ is a minimal element and $p_k$ a maximal one. A poset is \defn{pure} if all maximal chains have the same length; the common length is then called the \defn{length} of the poset.

A poset determines a simplicial set whose $k$-simplices are the non-strict chains $p_0 \leq \cdots \leq p_k$ in $P$. The \defn{classifying space} $BP$ of $P$ is the geometric realisation of this simplicial set. If we view $P$ as a category with objects the elements and a (unique) morphism $p \to q$ whenever $p\leq q$, the above simplicial set is the \defn{nerve}, and $BP$ is the classifying space of the category in the usual sense, see \cite[\S 2]{quillen}.

Elements $p$ and $q$ are said to be in the same \defn{component} of $P$ if there is a sequence of elements $p=p_0, p_1,\ldots,p_k=q$ such that either $p_i \leq p_{i+1}$ or $p_i \geq p_{i+1}$ for each $i=0,\ldots,k-1$; equivalently if the $0$-simplices corresponding  to $p$ and $q$ are in the same component of the classifying space $BP$.

The classifying space is a rather crude invariant of $P$. For example, there is a homeomorphism $BP \cong BP^\text{op}$, and if each finite set of elements has an upper bound (or a lower bound) then the classifying space $BP$ is contractible by \cite[Corollary 2]{quillen} since $P$, considered as a category,  is filtered.

\subsection{t-structures}
\label{poset of t-structures}

We fix some notation. Let $\cat{C}$ be an additive category. We write $c\in \cat{C}$ to mean $c$ is an object of $\cat{C}$. We will use the term \defn{subcategory} to mean strict, full subcategory. When $S$ is a subcategory we write $\cat{S}^\perp$ for the subcategory on the objects
\[
\{ c \in \cat{C} \ : \ \Mor{\cat{C}}{s}{c} = 0 \ \forall s \in \cat{S} \}
\]
and similarly ${}^\perp\cat{S}$ for $\{ c \in \cat{C} \ : \ \Mor{\cat{C}}{c}{s} = 0 \ \forall s \in \cat{S} \}$. When $\cat{A}$ and $\cat{B}$ are subcategories of $\cat{C}$ we write $\cat{A}\cap \cat{B}$ for the subcategory on objects which lie in both $\cat{A}$ and $\cat{B}$.

Suppose $\cat{C}$ is triangulated, with shift functor $[1]$. Exact triangles in $\cat{C}$ will be denoted either by $a\to b\to c \to a[1]$ or by a diagram
\begin{center}
\begin{tikzcd}
a \ar{rr} && b \ar{dl} \\
& c \ar[dashed]{ul} &
\end{tikzcd}
\end{center}
where the dotted arrow denotes a map $c \to a[1]$. We will always assume that $\cat{C}$ is essentially small so that isomorphism classes of objects form a set. Given sets $E_i$ of objects for $i\in I$  let $\langle E_i\ |\ i\in I \rangle$ denote the ext-closed subcategory generated by objects isomorphic to an element  in some $E_i$. We will use the same notation when the $E_i$ are subcategories of $\cat{C}$.

\begin{definition}
A \defn{t-structure} on a triangulated category $\cat{C}$ is an ordered pair $\ts{D} = (\aisle{D}{0},\coaisle{D}{1})$ of subcategories, satisfying:
\begin{enumerate}
\item $\aisle{D}{0}[1] \tsleq \aisle{D}{0}$ and $\coaisle{D}{1}[-1] \tsleq \coaisle{D}{1}$;
\item $\Mor{\cat{C}}{d}{d'}=0$ whenever $d\in \aisle{D}{0}$ and $d'\in\coaisle{D}{1}$;
\item for any $c\in \cat{C}$ there is an exact triangle $d \to c \to d' \to d[1]$ with $d \in \aisle{D}{0}$ and $d'\in\coaisle{D}{1}$.
\end{enumerate}
We write $\aisle{D}{n}$ to denote the shift $\aisle{D}{0}[-n]$, and so on. The subcategory $\aisle{D}{0}$ is called  the \defn{aisle} and $\coaisle{D}{0}$ the \defn{co-aisle} of the t-structure. The intersection $\heart{D}=\coaisle{D}{0} \cap \aisle{D}{0}$ of the aisle and co-aisle is an abelian category  known as the \defn{heart} of the t-structure --- see \cite[Th\'eor\`eme 1.3.6]{bbd} or \cite[\S 10.1]{ks}.
\end{definition}
The exact triangle $d\to c \to d' \to d[1]$ is unique up to isomorphism. The first term determines a right adjoint to the inclusion $\aisle{D}{0} \hookrightarrow \cat{C}$ and the last term a left adjoint to the inclusion $\coaisle{D}{1} \hookrightarrow \cat{C}$.

A t-structure $\ts{D}$ is \defn{bounded} if any object of $\cat{C}$ lies in $\coaisle{D}{-n} \cap \aisle{D}{n}$ for some $n\in \N$.
\begin{quote}
{\em Henceforth, we will assume that all t-structures are bounded.}
\end{quote}
This has three important consequences. Firstly, a bounded t-structure is completely determined by its heart; the aisle is recovered as
\[\cat{D}^{\leq0}=
\langle \heart{D}, \cat{D}^{-1} , \cat{D}^{-2} , \ldots \rangle.
\]
Secondly, the inclusion $\heart{D} \hookrightarrow \cat{C}$ induces an isomorphism $K \heart{D} \cong K\cat{C}$ of Grothendieck groups. Thirdly, if $\heart{D} \tsleq \heart{E}$ are hearts of bounded t-structures then $\ts{D}=\ts{E}$.

Under our assumption that $\cat{C}$ is essentially small, there is a {\em set} of t-structures on $\cat{C}$ (because t-structures correspond to aisles, and the latter are uniquely specified by certain subsets of the set of isomorphism classes of objects). In contrast, \cite{MR2652219} shows that t-structures on the derived category of all abelian groups (not necessarily finitely-generated) form a proper class.
\begin{definition}
Let $\PT{C}$ be the poset of bounded t-structures on $\cat{C}$, ordered by inclusion of the aisles. Abusing notation write  $\ts{D} \tsleq \ts{E}$ to mean $\aisle{D}{0} \tsleq \aisle{E}{0}$.
\end{definition}
There is a natural action of $\Z$ on $\PT{C}$ given by shifting: we write $\ts{D}[n]$ for the t-structure $(\aisle{D}{-n},\coaisle{D}{-n+1})$. Note that $\ts{D}[1] \tsleq \ts{D}$, and not vice versa.

\subsection{Torsion structures and tilting}
\label{tilting}

The notion of torsion structure, also known as a torsion/torsion-free pair, is an abelian analogue of that of t-structure; the notions are related by the process of tilting.

\begin{definition}
A \defn{torsion structure} on an abelian category $\cat{A}$ is an ordered pair $\ts{T} = (\aisle{T}{0},\coaisle{T}{1})$ of subcategories satisfying
\begin{enumerate}
\item $\Mor{\cat{A}}{t}{t'} = 0$ whenever $t\in \aisle{T}{0}$ and $t'\in\coaisle{T}{1}$;
\item for any $a\in \cat{A}$ there is a short exact sequence $0\to t \to a \to t' \to 0$ with $t \in \aisle{T}{0}$ and $t'\in\coaisle{T}{1}$.
\end{enumerate}
The subcategory $\aisle{T}{0}$ is given by the \defn{torsion theory} of $\ts{T}$, and $\coaisle{T}{1}$ by the \defn{torsion-free theory}; the motivating example is the subcategories of torsion and torsion-free abelian groups.
\end{definition}

The short exact sequence $0\to t\to a \to t' \to 0$ is unique up to isomorphism. The first term determines a right adjoint to the inclusion $\aisle{T}{0} \hookrightarrow \cat{A}$ and the last term a left adjoint to the inclusion $\coaisle{T}{1} \hookrightarrow \cat{A}$. It follows that $\aisle{T}{0}$ is closed under factors, extensions and coproducts and that $\coaisle{T}{1}$ is closed under subobjects, extensions and products. Torsion structures in $\cat{A}$, ordered by inclusion of their torsion theories, form a poset. It is bounded, with minimal element $(0,\cat{A})$ and maximal element $(\cat{A},0)$.
\begin{proposition}[{\cite[Proposition 2.1]{Happel:1996uq}, \cite[Theorem 3.1]{MR2327478}}]
\label{HRS}
Let  $\ts{D}$ be a t-structure on  a triangulated category $\cat{C}$.  Then there is a canonical isomorphism between the poset of torsion structures in the heart $\heart{D}$ and the interval $\its{\ts{D}}{\ts{D}[-1]}$ in $\PT{C}$ consisting of t-structures $\ts{E}$ with $\ts{D} \tsleq \ts{E} \tsleq \ts{D}[-1]$.
\end{proposition}

Let $\ts{D}$ be a t-structure on a triangulated category $\cat{C}$. It follows from Proposition~\ref{HRS} that a torsion structure $\ts{T}$ in the heart $\heart{D}$ determines a new t-structure
\[
\lt{D}{\cat{T}} = \left( \langle \aisle{D}{0}, \aisle{T}{1} \rangle , \langle \coaisle{T}{2}, \coaisle{D}{2} \rangle \right)
\]
called the \defn{left tilt} of $\ts{D}$ at $\ts{T}$, where by definition $\aisle{T}{k} = \aisle{T}{0}[-k]$ and similarly $\coaisle{T}{k} = \coaisle{T}{1}[1-k]$. The heart of the left tilt is $\langle \aisle{T}{1}, \coaisle{T}{1}\rangle$ and $\ts{D} \tsleq \lt{D}{\cat{T}} \tsleq \ts{D}[-1]$. The shifted t-structure $\rt{D}{\cat{T}} = \lt{D}{\cat{T}}[1]$ is called the \defn{right tilt} of $\ts{D}$ at $\ts{T}$. It has heart $\langle \aisle{T}{0}, \coaisle{T}{0}\rangle$ and $\ts{D}[1] \tsleq \rt{D}{\cat{T}} \tsleq \ts{D}$. Left and right tilting are inverse to one another: $\left( \coaisle{T}{1}, \aisle{T}{1}\right)$ is a torsion structure on the heart of $\lt{D}{\cat{T}}$, and right tilting with respect to this we recover the original t-structure. Similarly, $\left( \coaisle{T}{0}, \aisle{T}{0}\right)$ is a torsion structure on the heart of $\rt{D}{\cat{T}}$, and left tilting with respect to this we return to $\ts{D}$. Since there is a correspondence between bounded t-structures and their hearts we will, where convenient, speak of the left or right tilt of a heart.

\begin{definition}

Let the \defn{tilting poset} $\tilt{C}$ be the poset of t-structures with $\ts{D} \tileq \ts{E}$ if and only if there is a finite sequence of left tilts from $\ts{D}$ to $\ts{E}$.
\end{definition}
\begin{remark}
\label{tilting poset rmks}
An easy induction shows that if $\ts{D} \tileq \ts{E}$ then $\ts{D} \tsleq \ts{E} \tsleq \ts{D}[-k]$ for some $k\in \N$.

It follows that the identity on elements is a map of posets $\tilt{C} \to \PT{C}$. By Proposition~\ref{HRS}, if $\ts{D} \tsleq \ts{E} \tsleq \ts{D}[-1]$ then $\ts{D} \tileq \ts{E} \iff \ts{D} \tsleq \ts{E}$, so that the map induces an isomorphism $\itilt{\ts{D}}{\ts{D}[-1]} \cong \its{\ts{D}}{\ts{D}[-1]}$.
\end{remark}

\begin{lemma}
\label{tilting bound}
Suppose $\ts{D}$ and $\ts{E}$ are in the same component of $\tilt{C}$. Then $\ts{F} \tileq \ts{D},\ts{E} \tileq \ts{G}$ for some $\ts{F},\ts{G}$. (We do not claim that $\ts{F}$ and $\ts{G}$ are the infimum and supremum, simply that lower and upper bounds exist.)
\end{lemma}
\begin{proof}
If $\ts{D}$ and $\ts{E}$ are left tilts of some t-structure $\ts{H}$ then they are right tilts of $\ts{H}[-1]$, and vice versa. It follows that we can replace an arbitrary sequence of left and right tilts connecting $\ts{D}$ with $\ts{E}$ by a sequence of left tilts followed by a sequence of right tilts, or vice versa.
\end{proof}

\subsection{Algebraic t-structures}
\label{algebraic t-str}

We say an abelian category is \defn{algebraic} if it is a length category with finitely many isomorphism classes of simple objects. To spell this out, this means it is both Artinian and Noetherian so that every object has a finite composition series. By the Jordan-H\"older theorem, the graded object associated to such a composition series is unique up to isomorphism.
For instance, the module category $\operatorname{mod}A$ of a finite-dimensional algebra $A$ is algebraic.

The classes of the simple objects in an algebraic abelian category form a basis for the Grothendieck group, which is isomorphic to $\Z^n$, where $n$ is the number of such classes. A t-structure $\ts{D}$ is \defn{algebraic} if its heart $\heart{D}$ is. If $\cat{C}$ admits an algebraic t-structure then the heart of any other t-structure on $\cat{C}$ which is a length category must also have exactly $n$ isomorphism classes of simple objects, and therefore must be algebraic, since the two hearts have isomorphic Grothendieck groups.

Let  the \defn{algebraic tilting poset} $\tiltalg{C}$ be  the poset  consisting of the algebraic t-structures, with $\ts{D} \atileq \ts{E}$ when $\ts{E}$ is obtained from $\ts{D}$ by a finite sequence of left tilts, via  algebraic t-structures. Clearly
\[
\ts{D} \atileq \ts{E} \ \Rightarrow\ \ts{D} \tileq \ts{E} \ \Rightarrow\ \ts{D} \tsleq \ts{E},
\]
and there is an injective map of posets $\tiltalg{C} \to \tilt{C}$.
\begin{remark}
\label{BPP1}
There is an alternative algebraic description of $\tiltalg{C}$ when $\cat{C} = \cat{D}(A)$ is the bounded derived category of a finite-dimensional algebra $A$, of finite global dimension, over an algebraically-closed field. By \cite[Lemma 4.1]{pauk2} the poset $\mathbb{P}_1(\cat{C})$ of silting subcategories in $\cat{C}$ is the sub-poset of $\PT{C}^\text{op}$ consisting of the algebraic t-structures, and under this identification silting mutation in $\mathbb{P}_1(\cat{C})$ corresponds to (admissible) tilting in $\PT{C}^\text{op}$. Moreover, it follows from \cite[\S2.6]{MR2927802} that the partial order in $\mathbb{P}_1(\cat{C})$ is generated by silting mutation, so that $\ts{D} \tsleq \ts{E}\iff \ts{D} \atileq \ts{E}$ for algebraic $\ts{D}$ and $\ts{E}$. Hence $\tiltalg{C} \cong \mathbb{P}_1(\cat{C})^\text{op}$.

If $A$ does not have finite global dimension, then a similar result holds but we must replace the poset of silting subcategories in $\cat{C}$, with the analogous poset in the bounded homotopy category of finitely-generated projective modules.
\end{remark}
\begin{lemma}
\label{alg implies comparable}
Suppose $\ts{D}$ and $\ts{E}$ are t-structures and that $\ts{E}$ is algebraic. Then $\ts{E} \tsleq \ts{D}[-d]$ for some $d\in \N$.
\end{lemma}
\begin{proof}
Since $\ts{D}$ is bounded each simple object $s$ of the heart $\heart{E}$ is in $\aisle{D}{k_s}$ for some $k_s\in \Z$. Then $\heart{E} \tsleq \aisle{D}{d}$ for $d = \max_s \{k_s\}$ --- the maximum exists since there are finitely many simple objects in $\heart{E}$ --- and this implies $\ts{D} \tsleq \ts{D}[-d]$.
\end{proof}
\begin{remark}
It follows that $B\PT{C}$ is contractible whenever $\cat{C}$ admits an algebraic t-structure. To see this let $T_N(\cat{C})$ for $N\in \N$ be the sub-poset on $\{\ts{D} \mid \ts{E}[N] \tsleq \ts{D}\}$.  Note that $BT_N(\cat{C})$ is the cone on the vertex corresponding to $\ts{E}[N]$, hence is contractible. The above lemma implies that $B\PT{C}$ is the colimit of the diagram
\[
BT_0(\cat{C})\hookrightarrow BT_1(\cat{C})\hookrightarrow BT_2(\cat{C}) \hookrightarrow \cdots
\]
of contractible spaces. Hence it is also contractible.
\end{remark}

\begin{lemma}
\label{alg tilting bound}
Suppose $\ts{D}$ and $\ts{E}$ are in the same component of $\tiltalg{C}$. Then $\ts{F} \atileq \ts{D},\ts{E} \atileq \ts{G}$ for some $\ts{F},\ts{G}$ in that component.
\end{lemma}
\begin{proof}
This is proved in exactly the same way as Lemma~\ref{tilting bound}; note that all t-structures encountered in the construction will be algebraic.
\end{proof}
 It is not clear that the poset $\PT{C}$ of t-structures is always a lattice ---  see \cite{MR3135695} for an example in which the naive meet (i.e. intersection) of t-structures is not itself a t-structure, and also \cite{MR3092711} ---  and we do not claim that the lower and upper bounds of the previous lemma are infima or suprema. We do however have the following weaker result.
\begin{lemma}
\label{local lattice}
Suppose $\ts{D}$ is algebraic (in fact it suffices for its heart to be a length category). Then for each $\ts{D} \tsleq \ts{E}, \ts{F} \tsleq \ts{D}[-1]$ there is a supremum $\ts{E} \vee \ts{F}$ and an infimum $\ts{E} \wedge \ts{F}$ in $\PT{C}$.
\end{lemma}
\begin{proof}
We construct only the supremum $\ts{E} \vee \ts{F}$, the infimum is constructed similarly. We claim that $\langle \aisle{E}{0} ,\aisle{F}{0} \rangle$ is the aisle of a bounded t-structure; it is clear that this t-structure must then be the supremum in $\PT{C}$.

Since $\ts{D} \tsleq \ts{E}, \ts{F} \tsleq \ts{D}[-1]$ we may work with the corresponding torsion structures $\ts{T}_\ts{E}$ and $\ts{T}_\ts{F}$ on $\heart{D}$, and show that $ \aisle{T}{0} =\langle \aisle{T}{0}_\ts{E} , \aisle{T}{0}_\ts{F} \rangle$ is a torsion theory, with associated torsion-free theory $\coaisle{T}{1} = \coaisle{T}{1}_\ts{E}\cap \coaisle{T}{1}_\ts{F}$. Certainly $\Mor{\cat{C}}{t}{t'} = 0$ whenever $t\in \aisle{T}{0}$ and $t'\in\coaisle{T}{1}$, so it suffices to show that any $d\in \heart{D}$ sits in a short exact sequence $0 \to t \to d \to t' \to 0$ with $t\in \aisle{T}{0}$ and $t'\in\coaisle{T}{1}$. We do this in stages, beginning with the short exact sequence
\[
0 \to e_0 \to d \to e_0' \to 0
\]
with $e_0\in \aisle{T}{0}_\ts{E}$ and $e_0'\in\coaisle{T}{1}_\ts{E}$. Combining this with the short exact sequence $0 \to f_0 \to e_0' \to f_0' \to 0$ with $f_0\in \aisle{T}{0}_\ts{F}$ and $f_0'\in\coaisle{T}{1}_\ts{F}$ we obtain a second short exact sequence
\[
0 \to t \to d \to f_0' \to 0
\]
where $t$ is an extension of $e_0$ and $f_0$, and hence is in $\aisle{T}{0}$. Repeat this process, at each stage using the expression of the third term as an extension via alternately the torsion structures $\ts{T}_\ts{E}$ and $\ts{T}_\ts{F}$. This yields successive short exact sequences, each with middle term $d$ and first term in $ \aisle{T}{0}$, and such that the third term is a quotient of the third term of the previous sequence. Since $\heart{D}$ is a length category this process must stabilise. It does so when the third term has no subobject in either $\aisle{T}{0}_\ts{E}$ or $\aisle{T}{0}_\ts{F}$, \ie when the third term is in $\coaisle{T}{1}_\ts{E}\cap \coaisle{T}{1}_\ts{F} =\coaisle{T}{1}$. This exhibits the required short exact sequence and completes the proof.
\end{proof}

In general, this cannot be used inductively to show that the components of $\tiltalg{C}$ are lattices, since  $\ts{E} \wedge \ts{F}$ and $\ts{E} \vee \ts{F}$ might not be algebraic. For the remainder of this section we impose an assumption that guarantees that they are: let $\tilto{C}=\tiltalgo{C}$ be a component of the tilting poset consisting entirely of algebraic t-structures, equivalently a component of $\tiltalg{C}$ closed under all tilts.

\begin{lemma}
\label{lattice}
The component $\tilto{C}$ is a lattice. Infima and suprema in $\tilto{C}$ are also infima and suprema in $\PT{C}$.
\end{lemma}
\begin{proof}
Suppose $\ts{E},\ts{F} \in \tilto{C}$.  As in Lemma~\ref{tilting bound} we can replace an arbitrary sequence of left and right tilts connecting $\ts{E}$ with $\ts{F}$ by one consisting of a sequence of left tilts followed by a sequence of right tilts, or vice versa, but now using the infima and suprema of Lemma~\ref{local lattice} at each stage of the process. We can do this since $\tilto{C}$ consists entirely of algebraic t-structures, and therefore these infima and suprema are algebraic. Thus $\ts{E}$ and $\ts{F}$ have upper and lower bounds in $\tilto{C}$.

We now construct the infimum and supremum. First, convert the sequence of tilts from $\ts{E}$ to $\ts{F}$ into one of right followed by left tilts by the above process. Then if $\ts{E}, \ts{F} \tsleq \ts{G}$ the same is true for each t-structure along the new sequence. Now convert this new sequence to one of left tilts followed by right tilts, again by the above process. Inductively applying Lemma~\ref{local lattice} shows that each t-structure in the resulting sequence is still bounded above in $\PT{C}$ by $\ts{G}$. In particular the t-structure $\ts{H}$ reached after the final left tilt, and before the first right tilt, satisfies $\ts{E},\ts{F} \atileq \ts{H} \tsleq \ts{G}$. It follows that $\ts{H}\in \tilto{C}$ is the supremum $\ts{E} \vee \ts{F}$ of $\ts{E}$ and $\ts{F}$ in $\PT{C}$.

To complete the proof we need to show that $\ts{E} \vee \ts{F} \atileq \ts{G}$ whenever $\ts{G}$ is in $\tilto{C}$ and $\ts{E}, \ts{F} \atileq \ts{G}$. This follows since $\ts{E} \vee \ts{F} \atileq (\ts{E} \vee \ts{F})\vee \ts{G} = \ts{G}$.

The argument for the infimum is similar.
\end{proof}

\begin{lemma}
\label{finite intervals}
The following are equivalent:
\begin{enumerate}
\item Intervals of the form $\ialg{\ts{D}}{\ts{D}[-1]}$ in $\tilto{C}$ are finite.
\item All closed bounded intervals in $\tilto{C}$ are finite.
\end{enumerate}
\end{lemma}
\begin{proof}
Assume that intervals of the form $\ialg{\ts{D}}{\ts{D}[-1]}$ in $\tiltalg{C}$ are finite. Given $\ts{D} \atileq \ts{E}$ in $\tilto{C}$ recall that $\ts{E} \tsleq \ts{D}[-d]$ for some $d\in \N$ by Lemma~\ref{alg implies comparable}, so that
\[
\ts{D} \atileq \ts{E} \atileq \ts{E} \vee \ts{D}[-d] =\ts{D}[-d].
\]
Hence it suffices to show that intervals of the form $\ialg{\ts{D}}{\ts{D}[-d]}$ are finite. We prove this by induction on $d$. The case $d=1$ is true by assumption. Suppose it is true for $d<k$.  In diagrams it  will be convenient to use the notation $\ts{E} \rightsquigarrow \ts{F}$ to mean $\ts{F}$ is a left tilt of $\ts{E}$.

By definition of $\tiltalg{C}$ any element of the interval $\ialg{\ts{D}}{\ts{D}[-k]}$ sits in a chain of tilts $\ts{D} = \ts{D}_0 \rightsquigarrow \ts{D}_1 \rightsquigarrow \cdots \rightsquigarrow \ts{D}_r = \ts{D}[-k]$ via algebraic t-structures. This can be extended to a diagram
\begin{center}
\begin{tikzcd}
\ts{D} = \ts{D}_0 \arrow[rightsquigarrow]{r}\arrow[rightsquigarrow]{dr} & \ts{D}_1 \arrow[rightsquigarrow]{r} \arrow[rightsquigarrow]{d} &\ts{D}_2 \arrow[rightsquigarrow]{r} \arrow[rightsquigarrow]{d} & \cdots \arrow[rightsquigarrow]{r} & \ts{D}_{r-1} \arrow[rightsquigarrow]{r}\arrow[rightsquigarrow]{d} & \ts{D}_r =\ts{D}[-k] \\
& \ts{D}_1' \arrow[rightsquigarrow]{r} & \ts{D}_2' \arrow[rightsquigarrow]{r} & \cdots \arrow[rightsquigarrow]{r} &  \ts{D}_{r-1}' \arrow[rightsquigarrow]{ur}
\end{tikzcd}
\end{center}
of algebraic t-structures and tilts, where $\ts{D}_1'=\ts{D}[-1]$, so that $\ts{D}_1 \rightsquigarrow \ts{D}_1'$ as shown, and $\ts{D}_i' = \ts{D}_i \vee \ts{D}_{i-1}'$ is constructed inductively. The only point that requires elaboration is the existence of the tilt $\ts{D}_{r-1}' \rightsquigarrow \ts{D}_r$. First note that $\ts{D}_1', \ts{D}_2 \atileq \ts{D}_r$ so that $\ts{D}_2' = \ts{D}_2 \vee \ts{D}_{1}'\atileq \ts{D}_r$ too. By induction  $\ts{D}_{r-1}' \atileq \ts{D}_r$. Since
\[
\ts{D}_r[1] \atileq \ts{D}_{r-1} \atileq \ts{D}_{r-1}' \atileq \ts{D}_r
\]
$\ts{D}_r$ is a left tilt of $\ts{D}_{r-1}'$  by Proposition~\ref{HRS}.

The existence of the above diagram shows that each element of the interval $\ialg{\ts{D}}{\ts{D}[-k]}$ is a right tilt of some element of the interval $\ialg{\ts{D}[-1]}{\ts{D}[-k]}$. By induction the latter has only finitely many elements, and by assumption each of these has only finitely many right tilts. This  establishes the first implication. The converse is obvious.
\end{proof}

\subsection{Simple tilts}
\label{simple tilts}

Suppose $\ts{D}$ is an algebraic t-structure. Then each simple object $s\in \heart{D}$  determines two torsion structures on the heart, namely $(\langle s \rangle , \langle s \rangle^\perp)$ and $({}^\perp\langle s \rangle, \langle s \rangle)$. These are respectively minimal and maximal non-trivial torsion structures in $\heart{D}$. We say the left tilt at the former, and the right tilt at the latter, are \defn{simple}. We use the abbreviated notation $\lt{D}{s}$ and $\rt{D}{s}$ respectively for these tilts.

More generally we have the following notions. A torsion structure $\ts{T}$ is \defn{hereditary} if $t\in\aisle{T}{0}$ implies all subobjects of $t$ are in $\aisle{T}{0}$. It is \defn{co-hereditary} if $t\in \coaisle{T}{1}$ implies all quotients of $t$ are in $\coaisle{T}{1}$. It follows that the aisle of a hereditary torsion, dually the coaisle of a cohereditary torsion structure, are Serre subcategories. When $\ts{T}$ is a torsion structure on an algebraic abelian category then the hereditary torsion structures are those of the form $(S,S^\perp)$ where the torsion theory $S= \langle s_1,\ldots,s_k \rangle$ is generated by a subset of the simple objects. Dually, the co-hereditary torsion structures are those of the form $({}^\perp S,S)$. We use the abbreviated notation $\lt{D}{S}$  for the left tilt at $(S,S^\perp)$ and $\rt{D}{S}$ for the right tilt at $({}^\perp S,S)$. Note that, in the notation of the previous section, $\lt{D}{S}\wedge \lt{D}{S'} = \lt{D}{S \cap S'}$ and $\lt{D}{S} \vee \lt{D}{S'}  = \lt{D}{S \cup S'}$.

In general a tilt, even a simple tilt, of an algebraic t-structure need not be algebraic. However, if the heart is \defn{rigid}, \ie the simple objects have no self-extensions,  then
\cite[Proposition 5.4]{king-qiu} shows that the tilted t-structure is also algebraic. We will see later in Lemma~\ref{finite-type implies simple tilt length} that the same holds if the heart has only finitely many isomorphism classes of indecomposable objects.

\subsection{Stability conditions}
\label{stability conditions}

Let $\cat{C}$ be a triangulated category and $K\cat{C}$ be its Grothendieck group. A \defn{stability condition} $(\charge,\slicing)$ on $\cat{C}$ \cite[Definition 1.1]{MR2373143} consists of a group homomorphism $\charge \colon K\cat{C} \to \C$ and full additive subcategories $\slicing(\varphi)$ of $\cat{C}$ for each $\varphi\in \R$ satisfying
\begin{enumerate}
\item if $c\in \slicing(\varphi)$ then $\charge(c) = m(c)\exp(i\pi \varphi)$ where $m(c) \in \R_{>0}$;
\item $\slicing(\varphi+1) = \slicing(\varphi)[1]$ for each $\varphi\in \R$;
\item if $c\in \slicing(\varphi)$ and $c' \in \slicing(\varphi')$ with $\varphi > \varphi'$ then $\mor{c}{c'}=0$;
\item\label{HN filtration} for each nonzero object $c\in \cat{C}$ there is a finite collection of triangles
\begin{center}
\begin{tikzcd}
0=c_0 \ar{r} & c_1 \ar{r} \ar{d} & \cdots \ar{r} & c_{n-1} \ar{r} & c_n=c \ar{d}\\
& b_1\ar[dashed]{ul} &&& b_n \ar[dashed]{ul}
\end{tikzcd}
\end{center}
with $b_j \in \slicing(\varphi_j)$ where $\varphi_1 > \cdots > \varphi_n$.
\end{enumerate}
The homomorphism $\charge$ is known as the \defn{central charge} and the objects of $\slicing(\varphi)$ are said to be \defn{semistable of phase $\varphi$}. The objects $b_j$ are known as the \defn{semistable factors} of $c$. We define  $\varphi^+(c)=\varphi_1$ and $\varphi^-(c)=\varphi_n$. The \defn{mass}  of $c$ is defined to be $m(c)=\sum_{i=1}^n m(b_i)$.

For an interval $(a,b)\subseteq \R$ we set $\slicing(a,b) = \langle c \in \cat{C} \ \colon\  \varphi(c) \in (a,b) \rangle$, and similarly for half-open or closed intervals. Each stability condition $\sigma$ has an associated bounded t-structure $\ts{D}_\sigma = \left(\slicing(0,\infty), \slicing(-\infty,0]\right)$ with heart $\heart{D}_\sigma =\slicing(0,1]$.  Conversely, if we are given a bounded t-structure on $\cat{C}$ together with a stability function on the heart with the Harder--Narasimhan property ---  the abelian analogue of property (\ref{HN filtration}) above ---  then this determines a stability condition on $\cat{C}$ \cite[Proposition 5.3]{MR2373143}.

A stability condition is \defn{locally-finite} if we can find $\epsilon >0$ such that the quasi-abelian category $\slicing(t-\epsilon, t+\epsilon)$, generated by semistable objects with phases in $(t-\epsilon, t+\epsilon)$, has finite length (see \cite[Definition 5.7]{MR2373143}). The set of locally-finite stability conditions can be topologised so that it is a, possibly infinite-dimensional, complex manifold, which we denote $\stab{C}$ \cite[Theorem 1.2]{MR2373143}. The topology arises from the (generalised) metric
$$
d(\sigma,\tau) = \sup_{0\neq c \in \cat{C}} \max\left( | \varphi_\sigma^-(c) - \varphi_\tau^-(c)| , | \varphi_\sigma^+(c) - \varphi_\tau^+(c)|, \left| \log \frac{m_\sigma(c)}{m_\tau(c)}\right| \right)
$$
which takes values in $[0,\infty]$. It follows that for fixed $0\neq c\in \cat{C}$ the mass $m_\sigma(c)$, and lower and upper phases $\varphi_\sigma^-(c)$ and $\varphi_\sigma^+(c)$ are continuous functions $\stab{C} \to \R$. The projection
\[
\pi \colon \stab{C} \to \mor{K\cat{C}}{\C} \colon (\charge, \slicing) \mapsto \charge
\]
 is a local homeomorphism.

The group ${\rm Aut}(\cat{C})$ of auto-equivalences acts continuously on the space $\stab{C}$ of stability conditions with an automorphism $\alpha$ acting by
\begin{gather}
(\charge, \slicing) \mapsto \left(\charge\circ \alpha^{-1}, \alpha(\slicing)\right).
\end{gather}
There is also a smooth right action of the universal cover $G$ of $GL_2^+\R$. An element $g\in G$ corresponds to a pair $(T_g,\theta_g)$ where $T_g$ is the projection of $g$ to $GL_2^+\R$ under the covering map and $\theta_g\colon\R \to \R$ is an increasing map with $\theta_g(t+1)=\theta_g(t)+1$ which induces the same map as $T_g$  on the circle $\R/2\Z = \R^2-\{0\} / \R_{>0}$. The action is given by
\begin{equation}
\label{G action}
(\charge, \slicing) \mapsto \left(T_g^{-1} \circ \charge, \slicing\circ \theta_g\right).
\end{equation}
(Here we think of the central charge as valued in $\R^2$.) This action preserves the semistable objects, and also preserves the Harder--Narasimhan filtrations of all objects. The subgroup consisting of pairs with $T$ conformal is isomorphic to $\C$ with $\lambda\in \C$ acting via
\[
(\charge, \slicing) \mapsto \left(\exp(-i\pi\lambda)\charge, \slicing( \varphi + \re\lambda)\right)
\]
\ie by rotating the phases and rescaling the masses of semistable objects. This action is free and preserves the metric. The action of $1\in \C$ corresponds to the action of the shift automorphism $[1]$.
\begin{lemma}
\label{rotating tilts}
For any $g\in G$ the t-structures $\ts{D}_{g\cdot \sigma}$ and $\ts{D}_\sigma$ are related by a finite sequence of tilts.
\end{lemma}
\begin{proof}
This is simple to verify directly by considering the way in which $G$ acts on phases. Alternatively, note that  $G$ is connected, so that $\sigma$ and $g\cdot \sigma$ are in the same component of $\stab{C}$. Hence by \cite[Corollary 5.2]{MR3007660} the t-structures $\ts{D}_\sigma$ and $\ts{D}_\tau$ are related by a finite sequence of tilts.
\end{proof}

\subsection{Cellular stratified spaces}
\label{cellular stratified spaces}

A CW-cellular stratified space, in the sense of \cite{tamaki}, is a generalisation of a CW-complex in which non-compact cells are permitted. In \S\ref{stab alg} we will show that (parts of) stability spaces have this structure, and use it to show their contractibility. Here, we recall the definitions and results we will require.

A \defn{$k$-cell structure} on a subspace $e$ of a topological space $X$ is a continuous map $\alpha \colon D \to X$ where $\text{int}(\mathsf{D}^k) \subseteq D \subseteq \mathsf{D}^k$ is a subset of the $k$-dimensional disk $\mathsf{D}^k\subset \R^k$ containing the interior, such that $\alpha(D)=\overline{e}$, the restriction of $\alpha$ to $\text{int}(\mathsf{D}^k)$ is a homeomorphism onto $e$, and $\alpha$ does not extend to a map with these properties defined on any larger subset of $\mathsf{D}^k$. We refer to $e$ as a \defn{cell} and to $\alpha$ as a \defn{characteristic map} for $e$.

\begin{definition}
A \defn{cellular stratification}\label{cellstrat} of a topological space $X$ consists of a filtration
\[
\emptyset = X_{-1} \subseteq X_0 \subseteq X_1 \subseteq \cdots \subseteq X_k \subseteq \cdots
\]
by subspaces, with $X= \bigcup_{k\in \N} X_k$,  such that $X_k - X_{k-1} = \bigsqcup_{\lambda \in \Lambda_k} e_\lambda$ is a disjoint union of $k$-cells for each $k\in \N$. A  \defn{CW-cellular stratification}\label{CWcellstrat} is a cellular stratification satisfying the further conditions that
\begin{enumerate}
\item the stratification is closure-finite, i.e.\ the boundary $\partial e = \overline{e}- e$ of any $k$-cell is contained in a union of finitely many lower-dimensional cells;
\item $X$ has the weak topology determined by the closures $\overline{e}$ of the cells in the stratification, i.e.\ a subset $A$ of $X$ is closed if, and only if, its intersection with each $\overline{e}$ is closed.
\end{enumerate}
\end{definition}
When the domain of each characteristic map is the entire disk then a CW-cellular stratification is  nothing but a CW-complex structure on  $X$. Although the collection of cells and characteristic maps is part of the data of a cellular stratified space we will suppress it from our notation for ease-of-reading. Since we never consider more than one stratification of any given topological space there is no possibility for confusion.

A cellular stratification is said to be \defn{regular} if each characteristic map is a homeomorphism, and \defn{normal} if the boundary of each cell is a union of lower-dimensional cells. A regular, normal cellular stratification induces cellular stratifications on the domain of the characteristic map of each of its cells. Finally, we say a CW-cellular stratification is \defn{regular and totally-normal} if it is regular, normal, and in addition for each cell $e_\lambda$ with characteristic map $\alpha_\lambda \colon D_\lambda \to X$ the induced cellular stratification of $\partial D_\lambda = D_\lambda - \text{int} (\mathsf{D}^k)$ extends to a regular CW-complex structure on $\partial \mathsf{D}^k$. (The definition of totally-normal CW-cellular stratification in \cite{tamaki} is more subtle, as it handles the non-regular case too, but it reduces to the above for regular stratifications. A regular CW-complex is totally-normal, but regularity alone does not even entail normality for a
 CW-cellular stratified space.) Any union of strata in a regular, totally-normal CW-cellular stratified space is itself a regular, totally-normal CW-cellular stratified space.

A normal cellular stratified space $X$ has a \defn{poset of strata} (or face poset) $P(X)$ whose elements are the cells, and where $e_\lambda \leq e_\mu \iff e_\lambda \subseteq \overline{e_\mu}$. When $X$ is a regular CW-complex there is a homeomorphism  from the classifying space $BP(X)$ to $X$. More generally,
\begin{theorem}[{\cite[Theorem 2.50]{tamaki}}]
\label{X=BPX}
Suppose $X$ is a regular, totally-normal CW-cellular stratified space. Then $BP(X)$ embeds in $X$ as a strong deformation retract, in particular there is a homotopy equivalence $X \simeq BP(X)$.
\end{theorem}

\section{Algebraic stability conditions}
\label{stab alg}
 We say a stability condition $\sigma$ is \defn{algebraic} if the corresponding t-structure $\ts{D}_\sigma$ is algebraic. Let  $\stabalg{C} \tsleq \stab{C}$ be the subspace of algebraic stability conditions.

Write $S_\ts{D} = \{ \sigma \in \stab{C} \ : \ \ts{D}_\sigma=\ts{D} \}$ for the set of stability conditions with associated t-structure $\ts{D}$. When $\ts{D}$ is algebraic, a stability condition in $S_\ts{D}$ is uniquely determined  by a choice of central charge in
\begin{equation}
\label{algebraic heart}
 \U_{-}=\{ r\exp(i\pi\theta) \in \C \ : \ r>0 \ \textrm{and}\  \theta\in (0,1] \}
\end{equation}
for each simple object in the heart \cite[Lemma 5.2]{MR2549952}. Hence, in this case, an ordering of the simple objects determines an isomorphism $S_\ts{D} \cong \left(  \U_{-} \right)^n$. In particular, if $\cat{C}$ has an algebraic t-structure then $\stabalg{C} \neq \emptyset$.

The action of $\mathrm{Aut}(\cat{C})$ on $\stab{C}$ restricts to an action on the subspace $\stabalg{C}$. In contrast $\stabalg{C}$ need not be preserved by the  action of $\C$ on $\stab{C}$. The action of $i\R\subseteq \C$ uniformly rescales the masses of semistable objects; this does not change the associated t-structure and so preserves $\stabalg{C}$. However, $\R\subseteq \C$ acts by rotating the phases of semistables. Thus the action of $\lambda\in \R$ alters the t-structure by a finite sequence of tilts, and can result in a non-algebraic t-structure. In fact, the union of orbits  $\C \cdot \stabalg{C}$ consists of those stability conditions $\sigma$  for which  $(\slicing_\sigma(\theta,\infty), \slicing_\sigma(-\infty,\theta])$ is an algebraic t-structure for some $\theta\in \R$.  The choice of $\theta=0$ for {\em the} associated t-structure is purely conventional. If we define
\[
\stabalgt{C} = \{ \sigma \in \stab{C} \ \colon\  (\slicing_\sigma(\theta,\infty), \slicing_\sigma(-\infty,\theta])\ \text{is algebraic}\}
\]
then there is a commutative diagram
\begin{center}
\begin{tikzcd}
\stabalg{C} \ar[hookrightarrow]{r} \ar{d} & \stab{C} \ar{d}{\sigma \mapsto \theta\cdot \sigma} \\
\stabalgt{C} \ar[hookrightarrow]{r} &  \stab{C}
\end{tikzcd}
\end{center}
in which the vertical maps are homeomorphisms. So $\stabalgt{C} $ is independent   up to homeomorphism of the choice of $\theta\in \R$, but the way in which it is embedded in $\stab{C}$ is not.

\begin{lemma}
\label{alg phase criterion}
Suppose $\stabalg{\cat{C}} \neq \emptyset$. Then the space of algebraic stability conditions is contained in the union of full components of $\stab{C}$, \ie those components locally homeomorphic to $\mor{K\cat{C}}{\C}$. A stability condition  $\sigma$ in a full component of $\stab{C}$ is algebraic if and only if $\slicing_\sigma(0,\epsilon) =\emptyset$ for some $\epsilon >0$.
\end{lemma}
\begin{proof}
The assumption that $\stabalg{C}\neq \emptyset$ implies that $K\cat{C} \cong \Z^n$ for some $n\in \N$. It follows from the description of $S_\ts{D}$ for algebraic $\ts{D}$ above that any component containing an algebraic stability condition is full.

Suppose $\ts{D}$ is algebraic. Then for any $\sigma\in S_\ts{D}$ the simple objects are semistable. Since there are finitely many simple objects there is one, $s$ say, with minimal phase $\varphi^\pm_\sigma(s)=\epsilon >0$. It follows that $\slicing_\sigma(0,\epsilon) =\emptyset$.

Conversely, suppose $\slicing_\sigma(0,\epsilon) =\emptyset$ for some stability condition $\sigma$ in a full component. Then the heart $\slicing_\sigma(0,1] = \slicing_\sigma(\epsilon,1]$. Since $1-\epsilon<1$ we can  apply \cite[Lemma 4.5]{MR2376815} to deduce that the heart of $\sigma$ is an abelian length category.  It follows that the heart has $n$ simple objects (forming a basis of $K\cat{C}$), and hence is algebraic.
\end{proof}

\begin{lemma}
\label{nonempty interior}
The interior of $S_\ts{D}$ is non-empty precisely when $\ts{D}$ is algebraic.
\end{lemma}
\begin{proof}
The explicit description of $S_\ts{D}$ for algebraic $\ts{D}$ above shows that the interior is non-empty in this case. Conversely, suppose $\ts{D}$ is not algebraic and $\sigma \in S_\ts{D}$. Then by Lemma~\ref{alg phase criterion} there are $\sigma$-semistable objects of arbitrarily small strictly positive phase. It follows that the $\C$-orbit through $\sigma$ contains a sequence of stability conditions not in $S_\ts{D}$ with limit $\sigma$. Hence $\sigma$ is not in the interior of $S_\ts{D}$. Since $\sigma$ was arbitrary the latter must be empty.
\end{proof}

\begin{corollary}
\label{cstabalg open}
The subset $\C \cdot \stabalg{C} \tsleq \stab{C}$ is open, and when non-empty consists of those stability conditions in full components of $\stab{C}$ for which the phases of semistable objects are not dense in $\R$.
\end{corollary}
\begin{proof}
Suppose $\stabalg{\cat{C}}\neq \emptyset$. Then $K\cat{C}\cong \Z^n$ for some $n$.
A stability condition $\sigma\in \C\cdot \stabalg{C}$ clearly lies in a component of $\stab{C}$ meeting $\stabalg{C}$, and hence in a full component. By Lemma~\ref{alg phase criterion}, if $\sigma$ is in a full component then $\sigma \in \C \cdot \stabalg{C}$ if and only if $\slicing_\sigma(t,t+\epsilon) =\emptyset$ for some $t\in \R$ and $\epsilon >0$, equivalently if and only if the phases of semistable objects are not dense in $\R$.

To see that $\C\cdot\stabalg{C}$ is open note that if $\sigma\in \C\cdot \stabalg{C}$ and  $d(\sigma, \tau)<\epsilon /4$ then $\slicing_\sigma(t+\epsilon/4,t+3\epsilon/4) =\emptyset$ and so $\tau \in  \C \cdot \stabalg{C}$ too.
\end{proof}

\begin{example}
 Let $X$ be a smooth complex projective algebraic curve with genus $g(X)>0$. Then the space $\Stab(X)$ of stability conditions on the bounded derived category of coherent sheaves on $X$ is a single orbit of the $G$-action  \eqref{G action} through the stability condition with associated heart the coherent sheaves, and central charge $\charge(\mathcal{E})=-\mathrm{deg}\, \mathcal{E} +i \rk \mathcal{E}$ ---  see \cite[Theorem 9.1]{MR2373143} for $g(X)=1$ and \cite[Theorem 2.7]{MR2335991} for $g(X)>1$. It follows from the fact that there are semistable sheaves of any rational slope when $g(X)>0$ that the phases of semistable objects are dense for every stability condition in $\Stab(X)$. Hence $\Stab_\text{alg}(\cat{D}(X)) =\emptyset$. In fact this is true quite generally, since for `most' varieties the Grothendieck group $K(X)=K(\cat{D}(X))\not \cong \Z^n$.
 \end{example}

\begin{example}
\label{quiver density}
Let $Q$ be a finite connected quiver, and $\Stab(Q)$ the space of stability conditions on the bounded derived category of its finite-dimensional representations over an algebraically-closed field. When $Q$ has underlying graph of ADE Dynkin type, the phases of semistable objects form a discrete set \cite[Lemma 3.13]{MR3289326}; when it has extended ADE Dynkin type, the phases either form a discrete set or have accumulation points $t+\Z$ for some $t\in \R$ (all cases occur) \cite[Corollary 3.15]{MR3289326}; for any other  acyclic $Q$ there exists a  family of stability conditions for which the phases are dense in some non-empty open interval \cite[Proposition 3.32]{MR3289326}; and for $Q$ with oriented loops there exist stability conditions for which the phases of semistable objects are dense in $\R$ by \cite[Remark 3.33]{MR3289326}. It follows that $\Stab_\text{alg}(Q) = \Stab(Q)$ only in the Dynkin case; that $\C\cdot \Stab_\text{alg}(Q) = \Stab(Q)$ in the Dynkin or extended Dynkin cases; and that $\C\cdot \Stab_\text{alg}(Q) \neq \Stab(Q)$ when $Q$ has oriented loops. For a general acyclic quiver, we do not know whether $\C\cdot \Stab_\text{alg}(Q)= \Stab(Q)$ or not.
\end{example}

\begin{remark}
\label{density of phases}
The density of the phases of semistable objects for a stability condition is an important consideration in other contexts too. \cite[Proposition 4.1]{MR3007660} states that if phases for $\sigma$ are dense in $\R$ then the orbit of the universal cover $G$ of $GL_2^+(\R)$ through $\sigma$ is free, and the induced metric on the quotient $G \cdot \sigma /\C \cong G/\C \cong \U$ of the orbit is half the standard hyperbolic metric.
\end{remark}

\begin{lemma}
\label{cstabalg closed}
Suppose there exists a uniform lower bound on the maximal phase gap of algebraic stability conditions, \ie that there exists $\delta>0$ such that for each $\sigma \in \stabalg{C}$ there exists $\varphi\in \R$ with $\slicing_\sigma(\varphi-\delta,\varphi+\delta)=\emptyset$. Then $\C\cdot\stabalg{C}$ is closed, and hence is a union of components of $\stab{C}$.
\end{lemma}
\begin{proof}
Suppose $\sigma \in \overline{\C\cdot\stabalg{C}} - \C\cdot \stabalg{C}$. Let $\sigma_n\to \sigma$ be a sequence in $\C\cdot\stabalg{C}$ with limit $\sigma$. Write $\varphi_n^\pm$ for $\varphi_{\sigma_n}^\pm$ and so on.

Fix $\epsilon > 0$. There exists $N\in \N$ such that $d(\sigma_n,\sigma)<\epsilon$ for $n\geq N$. By Corollary~\ref{cstabalg open} the phases of semistable objects for $\sigma$ are dense in $\R$. Thus, given $\varphi \in \R$, we can find $\theta$ with $|\theta-\varphi|<\epsilon$ such that $\slicing_\sigma(\theta)\neq \emptyset$. So by \cite[\S3]{MR3007660} there exists $0\neq c\in \cat{C}$ such that $\varphi_n^\pm(c) \to \theta$. Hence $c\in \slicing_N(\theta-\epsilon,\theta+\epsilon) \tsleq  \slicing_N(\varphi-2\epsilon,\varphi+2\epsilon)$. In particular the latter is non-empty. Since $\varphi$ is arbitrary we obtain a contradiction by choosing $\epsilon<\delta/2$. Hence $\C\cdot \stabalg{C}$ is closed.
\end{proof}

\begin{example}
\label{P1 one}
Let $\Stab(\mathbb{P}^1)$ be the space of stability conditions on the bounded derived category $\cat{D}(\mathbb{P}^1)$ of coherent sheaves on $\mathbb{P}^1$. \cite[Theorem 1.1]{MR2219846} identifies  $\Stab(\mathbb{P}^1) \cong \C^2$. In particular there is a unique component, and it is full. The category $\cat{D}(\mathbb{P}^1)$ is equivalent to the bounded derived category $\cat{D}(\widetilde{A_1})$ of finite-dimensional representations of the Kronecker quiver $\widetilde{A_1}$. In particular, $\Stab_\text{alg}(\mathbb{P}^1)$ is non-empty. The Kronecker quiver has extended ADE Dynkin type, so by Example~\ref{quiver density} the phases of semistable objects for any $\sigma\in\Stab(\mathbb{P}^1)$ are either discrete or accumulate at the points $t+\Z$ for some $t\in \R$. The subspace $\Stab(\mathbb{P}^1) - \Stab_\text{alg}(\mathbb{P}^1)$ consists of those stability conditions with phases accumulating at $\Z \subseteq \R$. Therefore  $\C \cdot \Stab_\text{alg}(\mathbb{P}^1) = \Stab(\mathbb{P}^1)$ and $\Stab_\text{alg}(\P^1)$ is not closed. Neither is it  open \cite[p20]{MR2739061}: there are convergent sequences of stability conditions whose phases accumulate at $\Z$ such that the phase of each semistable object in the limiting stability condition is actually in $\Z$.

An explicit analysis of the semistable objects for each stability condition, as in \cite{MR2219846}, reveals that there is no lower bound on the maximum phase gap of algebraic stability conditions, so that whilst this condition is sufficient to ensure $\C\cdot\stabalg{C} =\stab{C}$ it is not necessary.
\end{example}

\subsection{The stratification of algebraic stability conditions}
\label{stratifying stab alg}

In this section we define and study a natural stratification of $\stabalg{C}$ with contractible strata. Suppose $\ts{D}$ is an algebraic t-structure on $\cat{C}$. Then $S_\ts{D} \cong (\U_{-})^n$ where $n=\rk(K\cat{C})$. For a subset $I$ of the simple objects in the heart $\heart{D}$ of $\ts{D}$ we define a subset of $\stab{C}$
\begin{align*}
S_{\ts{D},I} &= \{ \sigma \ \colon\  \ts{D}=\ts{D}_\sigma, \varphi_\sigma(s)=1 \ \text{for simple}\ s\in \heart{D} \iff  s \in I\}\\
&= \{ \sigma \ \colon\  \ts{D}=\ts{D}_\sigma,  \slicing_\sigma(1)=\langle I \rangle\}\\
&= \{ \sigma \ \colon\  \ts{D}= \left( \slicing_\sigma(0,\infty),  \slicing_\sigma(-\infty,0]\right),\  \lt{D}{I} = \left( \slicing_\sigma[0,\infty),  \slicing_\sigma(-\infty,0)\right)\}.
\end{align*}
Clearly $S_\ts{D} = \bigcup_I S_{\ts{D},I}$ and there is a decomposition
\begin{equation}
\label{stabalg decomp}
\stabalg{C} = \bigcup_{\ts{D}\ \text{alg}} S_\ts{D} = \bigcup_{\ts{D}\ \text{alg}} \Big(\bigcup_{I } S_{\ts{D},I}\Big).
\end{equation}
into strata of the form $S_{\ts{D},I}$. A choice of ordering of the simple objects of $\heart{D}$ determines a homeomorphism $S_\ts{D} \cong \left(  \U_{-} \right)^n$ under which the decomposition into strata corresponds to the the apparent decomposition of $\left(  \U_{-} \right)^n$ with $S_{\ts{D},I} \cong \U^{n-\# I} \times \R_{<0}^{\# I}$ where $\U$ is the strict upper half plane in $\C$. In particular each stratum $S_{\ts{D},I}$ is contractible.

Consider the closure $\overline{S_{\ts{D},I}}$ of a stratum. For  $I \subseteq K \subseteq \{s_1,\ldots,s_n\}$ let
\[
\partial_KS_{\ts{D},I} = \{ \sigma \in \overline{S_{\ts{D},I}} \ \colon\  \im \charge_\sigma(s)=0 \iff s\in K \},
\]
so that  $\overline{S_{\ts{D},I}} = \bigsqcup_{K} \partial_KS_{\ts{D},I}$ (as a set). For example $\partial_I S_{\ts{D},I} =S_{\ts{D},I}$.

\begin{lemma}
\label{components are strata}
For any t-structure $\ts{E}$, not necessarily algebraic, the intersection $S_\ts{E} \cap \partial_KS_{\ts{D},I}$ is a union of components of $\partial_KS_{\ts{D},I}$, \ie the heart of the stability condition remains constant in each component of $\partial_KS_{\ts{D},I}$. Each such component which lies in $\stabalg{C}$ is a stratum $S_{\ts{E},J}$ for some  $\ts{E}$ and subset $J$ of the simple objects in $\ts{E}$, with $\# J = \# K$.
\end{lemma}
\begin{proof}
Suppose $\sigma_n \to \sigma$ in $\stab{C}$. Then $\slicing_\sigma(0) = \langle 0\neq c\in \cat{C} \ \colon\  \varphi_n^\pm(c) \to 0\rangle$ by \cite[\S3]{MR3007660}. If $\sigma_n\in S_\ts{D}$ for all $n$ then
\[
\slicing_\sigma(0) = \Big\langle \{ 0\neq d\in \heart{D} \ \colon\  \varphi_n^+(d) \to 0\}, \{ 0\neq d\in \heart{D} \ \colon\  \varphi_n^-(d) \to 1\}[-1] \Big\rangle.
\]
Furthermore, $\ts{D}_\sigma$ is the right tilt of $\ts{D}$ at the torsion theory
\begin{equation}
\label{torsion theory}
\Big\langle  0\neq d\in \heart{D} \ \colon\  \varphi_n^-(d) \not \to 0  \Big\rangle = {}^\perp\Big \langle 0\neq d\in \heart{D} \ \colon\  \varphi_n^+(d) \to 0\Big\rangle.
\end{equation}
Now suppose $\sigma \in \partial_KS_{\ts{D},I}$ and $(\sigma_n)$ is a sequence in $S_{\ts{D},I}$ with limit $\sigma$. If $\varphi_n^+(d)\to 0$ for some $0\neq d\in \heart{D}$ then $\charge_n(d) \to \charge_\sigma(d) \in \R_{>0}$. Hence $d\in \langle K\rangle$. For $d\in \langle K\rangle$ there are three possibilities:
\begin{enumerate}
\item $\varphi_n^\pm(d) \to 0$ and $d\in \slicing_\sigma(0)$;
\item $\varphi_n^\pm(d) \to 1$ and $d\in \slicing_\sigma(1)$;
\item $\varphi_n^-(d) \to 0$, $\varphi_n^+(d) \to 1$, and $d$ is not $\sigma$-semistable.
\end{enumerate}
Since the upper and lower phases of $d$ are continuous in $\stab{C}$, and the possibilities are distinguished by discrete conditions on the limiting phases, we deduce that the torsion theory (\ref{torsion theory}) is constant for $\sigma$ in a component of $\partial_K S_{\ts{D},I}$. Hence the component is contained in $S_\ts{E}$ for some t-structure $\ts{E}$, and $S_\ts{E} \cap \partial_K S_{\ts{D},I}$ is a union of components of $\partial_K S_{\ts{D},I}$ as claimed.

Now suppose that $\sigma \in S_{\ts{E},J} \cap \partial_K S_{\ts{D},I}$ for some algebraic $\ts{E}$. On the one hand, $\langle J \rangle = \slicing_\sigma(1)$  since $\sigma\in S_{\ts{E},J}$, and therefore the triangulated closure of $J$ is $\slicing_\sigma(\Z) =  \langle \slicing_\sigma(\varphi) \ \colon\ \varphi\in \Z\rangle$. On the other hand, $\sigma \in \partial_K S_{\ts{D},I}$ implies that $\slicing_\sigma(\Z)$ is also the triangulated closure of the set $K$ of simple objects. The image of the map on Grothendieck groups induced by the inclusion $\slicing_\sigma(\Z) \hookrightarrow \cat{C}$ is therefore $\langle\, [t] \ \colon\  t\in J \rangle = \langle\, [s] \ \colon\  s\in K \rangle$. Since the elements of $J$ are simple objects in the heart of $\ts{E}$, and those of $K$ are simple objects in the heart of $\ts{D}$, and both $\ts{D}$ and $\ts{E}$ are algebraic by assumption, this is a free subgroup of rank $\#J = \#K$.

By a similar argument to that used for the first part of this proof
\[
\Big\langle  0\neq d\in \heart{D} \ \colon\  \varphi_n^-(d) \to 1 \Big\rangle
\]
is constant for $\sigma$ in a component of $\partial_K S_{\ts{D},I}$. It follows that $\slicing_\sigma(0)$ is constant in a component. By the first part $\ts{E}$ is fixed by the choice of component. As $\langle J \rangle =  \slicing_\sigma(1)=\slicing_\sigma(0)[1]$ the subset $J$ of simple objects in $\ts{E}$ is also fixed. So each component $A$ of $\stabalg{C} \cap \partial_K S_{\ts{D},I}$ is contained in some stratum $S_{\ts{E},J}$. The fact that we can perturb a stability condition by perturbing the charge allows us to deduce that $\partial_K S_{\ts{D},I}$ is a codimension $\#K$ submanifold of $\stab{C}$ and that $S_{\ts{E},J}$ is a codimension $\#J$ submanifold. Since $\#J =\#K$ the component $A$ must be an open subset of $S_{\ts{E},J}$. But directly from the definition of $\partial_K S_{\ts{D},I}$ one sees that the component $A$ is also a closed subset and, since $S_{\ts{E},J}$ is connected, we deduce that $A=S_{\ts{E},J}$ as required.
\end{proof}

\begin{corollary}
\label{bdy structure}
The decomposition (\ref{stabalg decomp}) of $\stabalg{C}$ satisfies the frontier condition, \ie if $S_{\ts{E},J} \cap \overline{S_{\ts{D},I}} \neq \emptyset$  then $S_{\ts{E},J} \tsleq \overline{S_{\ts{D},I}}$. In particular, the closure of each stratum is a union of lower-dimensional strata. Moreover,
\[
S_{\ts{E},J} \tsleq \overline{S_{\ts{D},I}} \quad \Rightarrow \quad \ts{E} \tileq \ts{D} \tileq \lt{D}{I} \tileq \lt{E}{J}.
\]
\end{corollary}
\begin{proof}
The frontier condition follows immediately from Lemma~\ref{components are strata}. Suppose that $S_{\ts{E},J} \tsleq \overline{S_{\ts{D},I}}$, and choose $\sigma$ in $S_{\ts{E},J}$. Let $\sigma_n \to \sigma$ where $\sigma_n \in S_{\ts{D},I}$. Then $\aisle{D}{0}=\slicing_n(0,\infty)$, $\aisle{D}{0}_I=\slicing_n[0,\infty)$, $\aisle{E}{0}=\slicing_\sigma(0,\infty)$, and $\aisle{E}{0}_J=\slicing_\sigma[0,\infty)$. Since $\slicing_n(0,\infty)$ and $\slicing_n[0,\infty)$ do not vary with $n$, and the minimal phase $\varphi_\tau^-(c)$ of any $0\neq c\in \cat{C}$ is continuous in $\tau$,
\[
\slicing_\sigma(0,\infty) \tsleq \slicing_n(0,\infty) \tsleq \slicing_n[0,\infty) \tsleq \slicing_\sigma[0,\infty),
\]
\ie $\ts{E} \tsleq \ts{D} \tsleq \lt{D}{I} \tsleq \lt{E}{J}$. Since all these t-structures are in the interval between $\ts{E}$ and $\ts{E}[-1]$ Remark~\ref{tilting poset rmks} implies that $\ts{E} \tileq \ts{D} \tileq \lt{D}{I} \tileq \lt{E}{J}$.
\end{proof}
\begin{lemma}
\label{frontier condition}
Suppose $\ts{D}$ and $\ts{E}$ are algebraic t-structures, and that $I$ and $J$ are subsets of simple objects in the respective hearts. If $\ts{E} \tileq \ts{D} \tileq \lt{D}{I} \tileq \lt{E}{J}$ then $S_{\ts{E},J} \tsleq \overline{S_{\ts{D},I}}$.
\end{lemma}
\begin{proof}
Fix $\sigma \in S_{\ts{E},J}$. Since $\ts{E}\tileq \ts{D} \tileq  \lt{E}{J}$ we know that $\ts{D} = \lt{E}{\cat{T}}$ for some torsion structure $\ts{T}$ on $\heart{E}$, and moreover that $\aisle{T}{0} \tsleq \langle J \rangle = \slicing_\sigma(1)$. Any simple object of $\heart{D}$ lies either in $\aisle{T}{0}[-1]$ or in $\coaisle{T}{1}$. Hence any simple object $s$ of $\heart{D}$ lies in $\slicing_\sigma[0,1]$, and $s\in \slicing_\sigma(0) \iff s\in \aisle{T}{0}[-1]$. Moreover, if $s\in I$ then $s[-1]\in \lt{D}{I}^{\leq 0} \tsleq \lt{E}{J}^{\leq 0} = \slicing_\sigma[0,\infty)$. Thus $s\in I \Rightarrow s\in \slicing_\sigma(1)$.

Since the simple objects of $\heart{D}$ form a basis of $K \cat{C}$ we can perturb $\sigma$ by perturbing their charges. Given $\delta>0$ we can always make such a perturbation to obtain a stability condition $\tau$ with $d(\sigma,\tau)<\delta$ for which $\charge_\tau(s) \in \U \cup \R_{>0}$ for all simple $s$ in $\heart{D}$, and $\charge_\tau(s)\in \R_{>0} \iff s\in\slicing_\sigma(0)$. We can then rotate, \ie act by some $\lambda\in \R$, to obtain a stability condition $\omega$ with $d(\tau,\omega)<\delta$ such that $\charge_\tau(s) \in \U$ for all simple $s$ in $\ts{D}$. We will prove that $\omega\in S_\ts{D}$. Since the perturbation and rotation can be chosen arbitrarily small it will follow that $\sigma\in \overline{S_\ts{D}}$. And since $s\in \slicing_\sigma(1)$ whenever $s\in I$ we can refine this statement to $\sigma \in \overline{S_{\ts{D},I}}$ as claimed.

It remains to prove $\omega\in S_\ts{D}$. For this it suffices to show that each simple $s$ in $\heart{D}$ is $\tau$-semistable. For then $s$ is $\omega$-semistable too, and the choice of $\charge_\omega$ implies that $s \in \slicing_\omega(0,1]$. The hearts of distinct (bounded) t-structures cannot be nested, so this implies $\ts{D} = \ts{D}_\omega$, or equivalently $\omega\in S_\ts{D}$  as required.

Since $\ts{E}$ is algebraic Lemma~\ref{alg phase criterion} guarantees that there is some $\delta>0$ such that $\slicing_\sigma(0,2\delta]=\emptyset$.  Provided $d(\sigma,\tau)<\delta$ we have
\[
\slicing_\sigma(0,1]= \slicing_\sigma(2\delta,1] \tsleq \slicing_\tau(\delta,1+\delta] \tsleq \slicing_\sigma(0,1+2\delta]= \slicing_\sigma(0,1].
\]
It follows that the Harder--Narasimhan $\tau$-filtration of any  $e\in \heart{E} = \slicing_\sigma(0,1]$ is a filtration by subobjects of $e$ in the abelian category $\slicing_\sigma(0,1]$.

Consider a simple $s'$ in $\heart{D}$ with $s'[1] \in \aisle{T}{0}$. Since $\aisle{T}{0}$ is a torsion theory any quotient of $s'[1]$ is also in $\aisle{T}{0}$, in particular the final factor in the Harder--Narasimhan $\tau$-filtration, $t$ say, is in $\aisle{T}{0}$. Hence $t[-1] \in \heart{D}$ and $[t] = - \sum m_s [s] \in K\cat{C}$ where the sum is over the simple $s$ in $\heart{D}$ and the $m_s\in \N$. Since $\im \charge_\tau(s) \geq 0$ for each simple $s$ it follows that $\im \charge_\tau(t) = -\sum m_s \im \charge_\tau(s) \leq 0$. Combined with the fact that $t$ is $\tau$-semistable with phase in $(\delta,1+\delta]$ we have $\varphi_\tau^-(s'[1]) =\varphi_\tau(t) \geq 1$. Hence $s'\in \slicing_\tau[1,1+\delta]$. But $s'[1] \in \aisle{T}{0}$ so $\charge_\tau(s'[1]) \in \R_{<0}$ and therefore $s'[1] \in \slicing_\tau(1)$, and in particular is $\tau$-semistable.

Now suppose $s'\in \coaisle{T}{1}$.  Since $\coaisle{T}{1}$ is a torsion-free theory in $\slicing_\sigma(0,1]$  any subobject of $s'$  is also in $\coaisle{T}{1}$. In contrast, $s'$ cannot have any {\em proper} quotients in $\coaisle{T}{1}$: if it did we would obtain a short exact sequence \[0\to f \to s \to f' \to 0\] in $\slicing_\sigma(0,1]$ with $f,f'\in \coaisle{T}{1}$. This would also be short exact in $\heart{D}$, contradicting the fact that $s'$ is simple. It follows that any proper quotient of $s'$ is in $\aisle{T}{0}$. The argument of the previous paragraph then shows that either $s'$ is $\tau$-semistable (with no proper semistable quotient), or $s'\in \slicing_\tau[1,1+\delta]$. But $\im \charge_\tau(s')>0$ so the latter is impossible, and $s'$ must be $\tau$-semistable. This completes the proof.
\end{proof}

\begin{definition}
\label{poset of intervals}
Let $\PI{C}$ be the poset whose elements are intervals in the poset $\tilt{C}$ of t-structures of the form $\itilt{\ts{D}}{\lt{D}{I}}$, where $\ts{D}$ is algebraic and $I$ is a subset of the simple objects in the heart of $\ts{D}$. We order these intervals by inclusion. We do not assume that $\lt{D}{I}$ is algebraic.
\end{definition}

\begin{corollary}
\label{poset-iso}
There is an isomorphism $\PI{C}^\textrm{op} \to P(\stabalg{C})$ of posets given by the correspondence $\itilt{\ts{D}}{\lt{D}{I}}\longleftrightarrow S_{\ts{D},I}$.  Components of $\stabalg{C}$ correspond to  components of $\tiltalg{C}$.
\end{corollary}
\begin{proof}
The existence of the isomorphism is direct from Corollary~\ref{bdy structure} and Lemma~\ref{frontier condition}. In particular, components of these posets are in $1$-to-$1$ correspondence. The second statement follows because components of $\stabalg{C}$ correspond to components of $P(\stabalg{C})$, and components of $\PI{C}$ correspond to components of $\tiltalg{C}$.
\end{proof}
\begin{remark}
\label{BPP2}
Following Remark~\ref{BPP1} we note an alternative description of $\PI{C}$ when $\cat{C}= \cat{D}(A)$ is the bounded derived category of a finite-dimensional algebra $A$ over an algebraically-closed field, and has finite global dimension. By \cite[Lemma 4.1]{pauk2}  $\PI{C}^\text{op}  \cup \{\hat{0}\} \cong \mathbb{P}_2(\cat{C})$ is the {\em poset of silting pairs} defined in \cite[\S3]{pauk2}, where $\hat{0}$ is a formally adjoined minimal element. Hence, by the above corollary, $P(\stabalg{C}) \cup \{\hat{0}\}\cong \mathbb{P}_2(\cat{C})$.
\end{remark}

\begin{remark}
\label{poset relations}
If $\ts{D}$ and $\ts{E}$ are not both algebraic then $\ts{D} \tileq \ts{E} \tileq \ts{D}[-1]$ need not imply $S_\ts{D} \cap \overline{S_\ts{E}} \neq \emptyset$, see \cite[p20]{MR2739061} for an example. Thus components of $\stabalg{C}$ may not correspond to components of $\tilt{C}$. In general we have maps
\begin{center}
\begin{tikzcd}
\pi_0 \stabalg{C}  \ar{r} \ar[equals]{d} & \pi_0 \stab{C} \ar{d} & \\
\pi_0 \tiltalg{C} \ar{r} & \pi_0 \tilt{C} \ar{r} & \pi_0 \PT{C}.
\end{tikzcd}
\end{center}
The bottom row is induced from the maps $\tiltalg{C} \to \tilt{C} \to \PT{C}$, the vertical equality holds by the above corollary, and the vertical map exists because $S_\ts{D}$ and $S_\ts{E}$ in the same component of $\stab{C}$ implies that $\ts{D}$ and $\ts{E}$ are related by  a finite sequence of tilts \cite[Corollary 5.2]{MR3007660}.
\end{remark}
\begin{lemma}
\label{single component}
Suppose that $\tiltalg{C}=\tilt{C}=\PT{C}$ are non-empty. Then $\stabalg{C}=\stab{C}$ has a single component.
\end{lemma}
\begin{proof}
It is clear that $\stab{C} = \stabalg{C}\neq \emptyset$. Let $\sigma,\tau\in \stab{C}$. Since $\tiltalg{C}=\tilt{C}$ the associated t-structures $\ts{D}_\sigma$ and $\ts{D}_\tau$ are algebraic, so that $\ts{D}_\sigma \tsleq \ts{D}_\tau[-j]$ for some $j\in\N$ by Lemma~\ref{alg implies comparable}. Since $\tiltalg{C}=\PT{C}$ this implies $\ts{D}_\sigma \atileq \ts{D}_\tau[-j]$, and thus $\ts{D}_\sigma$ and $\ts{D}_\tau$ are in the same component of $\tiltalg{C}$. Hence by Corollary~\ref{poset-iso} $\sigma$ and $\tau$ are in the same component of $\stabalg{C}=\stab{C}$.
\end{proof}
\begin{lemma}
\label{stabalg connected}
Suppose $\cat{C} = \cat{D}(A)$ for a finite-dimensional algebra $A$  over an algebraically-closed field, with finite global dimension. Then $\stabalg{C}$ is connected. Moreover, any component of $\stab{C}$ other than that containing $\stabalg{C}$ consists entirely of stability conditions for which the phases of semistable objects are dense in $\R$.
\end{lemma}
\begin{proof}
By Remark~\ref{BPP1} $\tiltalg{C}$ is the sub-poset of $\PT{C}$ consisting of the algebraic t-structures. The proof that $\stabalg{C}$ is connected is then the same as that of the previous result. For the last part note that if $\sigma$ is a stability condition for which the phases of semistable objects are not dense then acting on $\sigma$ by some element of $\C$ we obtain an algebraic stability condition. Hence $\sigma$ must be in the unique component of $\stab{C}$ containing $\stabalg{C}$.
\end{proof}
\begin{remark}
\label{connectedness in algebra case}
To show that $\stab{C}$ is connected when $\cat{C}=\cat{D}(A)$ as in the previous result it suffices to show that there are no stability conditions for which the phases of semistable objects are dense. For example, from Example~\ref{quiver density}, and the fact that the path algebra of an acyclic quiver is a finite-dimensional algebra of global dimension $1$, we conclude that $\Stab(Q)$ is connected whenever $Q$ is of ADE Dynkin, or extended Dynkin, type. (Later we show that $\Stab(Q)$ is contractible in the Dynkin case; it was already known to be simply-connected by \cite{MR3281136}.)

By Remark~\ref{density of phases}, the universal cover $G=\widetilde{GL_2^+(\R)}$ acts freely on a component consisting of stability conditions for which the phases are dense. In contrast, it does not act freely on a component containing algebraic stability conditions since any such contains stability conditions for which the central charge is real, and these have non-trivial stabiliser. Hence, the $G$-action also distinguishes the component containing $\stabalg{C}$ from the others, and if there is no component on which $G$ acts freely $\stab{C}$ must be connected.
\end{remark}

Suppose $\stabalg{C}\neq \emptyset$. Let $\mathrm{Bases}(K\cat{C})$ be the groupoid whose objects are pairs consisting of an ordered basis of the free abelian group $K\cat{C}$ and a subset of this basis, and whose morphisms are automorphisms relating these bases (so there is precisely one morphism in each direction between any two objects; we do not ask that it preserve the subsets). Fix an ordering of the simple objects in the heart of each algebraic t-structure. This fixes isomorphisms
\[
S_{\ts{D},I} \cong \U^{n-\# I} \times \R_{<0}^{\# I}.
\]
 Regard the poset $\PI{C}$ as a category, and let $F_\cat{C} \colon \PI{C} \to \mathrm{Bases}(K\cat{C})$ be the functor taking $\itilt{\ts{D}}{\lt{D}{I}}$ to the pair consisting of the ordered basis of classes of simple objects in $\ts{D}$ and the subset of classes of $I$. This uniquely specifies $F_\cat{C}$ on morphisms.
 \begin{proposition}
 \label{reconstruction}
 The functor $F_\cat{C}$ determines $\stabalg{C}$ up to homeomorphism as a space over $\mor{K\cat{C}}{\C}$.
 \end{proposition}
 \begin{proof}
 As sets there is a commutative diagram
\begin{center}
\begin{tikzcd}
\stabalg{C} \ar{rr}{\beta} \ar{dr}[swap]{\pi}&& \sum_{\ts{D},I} \U^{n-\# I} \times \R_{<0}^{\# I} \ar{dl}{\sum \pi_{\ts{D},I}} \\
& \mor{K\cat{C}}{\C} &
\end{tikzcd}
\end{center}
where the map $\pi_{\ts{D},I}$ is determined from the pair $F_\cat{C}\left(\itilt{\ts{D}}{\lt{D}{I}}\right)$ of basis and subset, and $\beta$ is defined using the bijections $S_{\ts{D},I} \cong \U^{n-\# I} \times \R_{<0}^{\# I}$. The subsets
\[
U_{\ts{E},J} = \bigcup_{\ts{E}\tileq \ts{D} \tileq \lt{D}{I} \tileq \lt{E}{J}} \pi_{\ts{D},I}^{-1}U,
\]
where $U$ is open in $\mor{K\cat{C}}{\C}$, form a base for a topology. With this topology, $\beta$ is a homeomorphism. To see this note that
\[
\beta^{-1}U_{\ts{E},J} =  \left(\bigcup_{\ts{E}\tileq \ts{D} \tileq \lt{D}{I} \tileq \lt{E}{J}} S_{\ts{D},I} \right)\cap \pi^{-1}U
\]
is the intersection of an open subset with an upward-closed union of strata, hence open. So $\beta$ is continuous. Moreover, all sufficiently small open neighbourhoods of a point of $\stabalg{C}$ have this form, so the bijection $\beta$ is an open map, hence a homeomorphism.
\end{proof}
A more practical approach is to study the homotopy-type of $\stabalg{C}$. In good cases this is encoded in the poset $P(\stabalg{C}) \cong \PI{C}^\text{op}$.

Recall that a stratification is \defn{locally-finite} if any stratum is contained in the closure of only finitely many other strata, and \defn{closure-finite} if the closure of each stratum is a union of finitely many strata.
\begin{lemma}
\label{local and closure finiteness}
The following are equivalent:
\begin{enumerate}
\item the stratification of $\stabalg{C}$ is locally-finite;
\item the stratification of $\stabalg{C}$ is closure-finite;
\item each interval $\ialg{\ts{D}}{\ts{D}[-1] }$ in $\tiltalg{C}$ is finite.
\end{enumerate}
\end{lemma}
\begin{proof}
This follows easily from Corollary~\ref{poset-iso} which states that $S_{\ts{E},J} \tsleq \overline{S_{\ts{D},I}} \iff \ts{E} \tileq \ts{D} \tileq \lt{D}{I} \tileq \lt{E}{J}$. Thus the size of the interval $\ialg{\ts{D}}{\ts{D}[-1]}$ is precisely
\[
 \#\{ \ts{E} \in \tiltalg{C} \ \colon\  \overline{S_\ts{E}} \cap S_\ts{D} \neq \emptyset \}
= \#\{ \ts{E} \in \tiltalg{C} \ \colon\  \overline{S_\ts{D}} \cap S_{\ts{E}[1]} \neq \emptyset \}.
\]
The result follows because each $S_\ts{D}$ is a finite union of strata, and each stratum is in some $S_\ts{D}$.
\end{proof}
\begin{proposition}
\label{stabalg  is cellular}
The space $\stabalg{C}$ of algebraic stability conditions, with the decomposition into the strata $S_{\ts{D},I}$, can be given the structure of a regular, normal cellular stratified space. It is a regular, totally-normal CW-cellular stratified space precisely when $\stabalg{C}$ is locally-finite.
\end{proposition}
\begin{proof}
 First we define a cell structure on $S_{\ts{D},I}$. Denote the projection onto the central charge by $\pi \colon \stab{C} \to \mor{K\cat{C}}{\C}$. Choose a basis for $K\cat{C}$ and identify $\mor{K\cat{C}}{\C} \cong \C^n \cong \R^{2n}$ with $2n$-dimensional Euclidean space. Note that
\[
\overline{S_{\ts{D},I}}\cap\stabalg{C} \cong \pi\left( \overline{S_{\ts{D},I}} \cap\stabalg{C} \right) \tsleq \overline{ \pi \left( S_{\ts{D},I} \right)}
\]
and that $\overline{ \pi \left( S_{\ts{D},I} \right)}$ is the real convex closed polyhedral cone
\[
C= \{ \charge \ \colon\   \im \charge(s) \geq 0\ \text{for}\ s\not \in I\ \text{and}\ \im \charge(s)=0,\ \re \charge(s) \leq 0\  \text{for } s\in I \}
\]
in $\mor{K\cat{C}}{\C}$. The projection $\pi$ identifies the stratum $S_{\ts{D},I}$ with the (relative) interior of $C$. By Corollary~\ref{bdy structure} $\overline{S_{\ts{D},I}} \cap \stabalg{C}$ is a union of strata. Moreover, the projection of each boundary stratum
\[
S_{\ts{E},J} \tsleq \overline{S_{\ts{D},I}} \cap \stabalg{C}
\]
 is cut out by a finite set of (real) linear equalities and inequalities. Therefore we can subdivide $C$ into a union of real convex polyhedral sub-cones in such a way that each stratum is identified with the (relative) interior of one of these sub-cones.

Let $A(1,2)$ be the open annulus in $\mor{K\cat{C}}{\C}$ consisting of points of distance in the range $(1,2)$ from the origin, and $A[1,2]$ its closure. Then we have a continuous map
\[
\overline{S_{\ts{D},I}}\cap\stabalg{C} \stackrel{\pi}{\longrightarrow} C - \{0\} \cong C\cap A(1,2) \hookrightarrow C\cap A[1,2]
\]
where $C - \{0\} $ is identified with $C\cap A(1,2)$ via a radial contraction. The subdivision of $C$ into cones induces the structure of a compact curvilinear polyhedron on the intersection $C\cap A[1,2]$. A choice of homeomorphism from $C\cap A[1,2]$ to a closed cell yields a map from $\overline{S_{\ts{D},I}}\cap\stabalg{C}$ to a closed cell which is a homeomorphism onto its image. The inverse from this image is a characteristic map for the stratum $S_{\ts{D},I}$, and the collection of these gives $\stabalg{C}$ the structure of a regular, normal cellular stratified space.

When the stratification of $\stabalg{C}$ is locally-finite the cellular stratification is closure-finite by Lemma~\ref{local and closure finiteness}, and any point is contained in the interior of a closed union of finitely many cells. This guarantees that $\stabalg{C}$ has the weak topology arising from the cellular stratification, which is therefore a CW-cellular stratification. We can also choose the above subdivision of $C$ to have  finitely many sub-cones. In this case the curvilinear polyhedron $C\cap A[1,2]$ has finitely many faces, and therefore has a CW-structure for which the strata of $\overline{S_{\ts{D},I}}\cap\stabalg{C}$ are identified with certain open cells. It follows that the cellular stratification is totally-normal. Conversely, if the stratification is CW-cellular then it is closure-finite, and hence by Lemma~\ref{local and closure finiteness} it is locally-finite.
\end{proof}

\begin{corollary}
\label{comb model for stab alg}
Suppose the stratification of $\stabalg{C}$ is locally-finite
and let $n = \rk(K\cat{C})$.
Then we have the following:
\begin{enumerate}
\item There is a homotopy equivalence $\stabalg{C} \simeq BP\left( \stabalg{C} \right)$.
\item $BP\left(\stabalg{C}\right)$ is a CW-complex of dimension $\leq n$
\item The integral homology groups $H_i\left(\stabalg{C}\right) = 0$ for $i > n$.
\end{enumerate}
\end{corollary}
\begin{proof}
The first claim is direct from Proposition~\ref{stabalg is cellular} and Theorem~\ref{X=BPX}.
By Corollary~\ref{comb model for stab alg} $\stabalg{C} \simeq BP\left(\stabalg{C}\right)$.  A chain in the poset $P\left(\stabalg{C}\right)$ consists of a sequence of strata of $\stabalg{C}$ of decreasing codimension, each in the closure of the next. Since the maximum codimension of any stratum is $n$, the length of any chain is less than or equal to $n$.  Hence $BP\left(\stabalg{C}\right)$ is a CW-complex of dimension $\leq n$, and the last claim also follows.
\end{proof}

\begin{remark}
\label{subspaces of stabalg}
If $\stabalg{C}$ is  locally-finite then any union $U$ of strata of $\stabalg{C}$ is a regular, totally-normal CW-cellular stratified space. Hence there is a homotopy equivalence $U \simeq BP(U)$ and $H_i(U)=0$ for $i>n=\rk(K\cat{C})$.
\end{remark}

\begin{example}
\label{P1 two}
We continue Example~\ref{P1 one}. The `Kronecker heart' \[\langle \mathcal{O}, \mathcal{O}(-1)[1] \rangle\] of $\cat{D}(\P^1)$ is algebraic. There are infinitely many torsion structures on this heart such that the tilt is a t-structure with heart isomorphic to the Kronecker heart \cite[\S3.2]{MR2739061}. It quickly follows from Corollary~\ref{poset-iso} that the stratification of $\Stab_\text{alg}(\P^1)$ is neither closure-finite nor locally-finite ---  see \cite[Figure 5]{MR2739061} for a diagram of the codimension $2$ strata in the closure of the stratum corresponding to the Kronecker heart.
\end{example}

\subsection{More on the poset of strata}
\label{poset of strata}

Corollary~\ref{comb model for stab alg} shows that if $\stabalg{C}$ is closure-finite  and locally-finite, then its homotopy-theoretic properties are encoded in the poset $P\left( \stabalg{C} \right)$. In the remainder of this section we elucidate some of the latter's good properties.

 The assumptions that $\stabalg{C}$ is locally-finite and closure-finite are respectively equivalent to the statements that the unbounded closed intervals $[S,\infty)$ and $(-\infty,S]$ are finite for each $S\in P\left( \stabalg{C} \right)$. It follows of course that closed bounded intervals are also finite, but in fact the latter holds without these assumptions.
\begin{lemma}
\label{closed interval lemma}
 Suppose $S_{\ts{E},J} \tsleq \overline{S_{\ts{D},I}}$. Then the closed interval $[S_{\ts{E},J},S_{\ts{D},I}]$ in $P\left(\stabalg{C}\right)$ is isomorphic to a sub-poset of  $[I,K]^\text{op}$. Here the subset  $K$ is uniquely determined by the requirement that $S_{\ts{E},J} \tsleq \partial_KS_{\ts{D},I}$, and subsets of the simple objects in $\heart{D}$ are ordered by inclusion.
\end{lemma}
\begin{proof}
Suppose $S_{\ts{E},J}\tsleq \partial_KS_{\ts{D},I}$ and fix $\sigma\in S_{\ts{E},J}$. Using the fact that $\stab{C}$ is locally isomorphic to $\mor{K\cat{C}}{\C}$ we can choose an open neighbourhood $U$ of $\sigma$ in $\stab{C}$ so that $U \cap \partial_LS_{\ts{D},I}$ is non-empty and connected for any subset $I \subseteq L \subseteq K$, and empty when $L \not \subseteq K$. It follows that $U$ meets a unique component of  $\partial_LS_{\ts{D},I}$ for each $I \subseteq L \subseteq K$. The strata in $[S_{\ts{E},J},S_{\ts{D},I}]$ correspond to those components for which the heart is algebraic. Since $\partial_LS_{\ts{D},I} \tsleq \overline{\partial_{L'} S_{\ts{D},I}} \iff L' \subseteq L$ the result follows.
\end{proof}

We have seen that $\stabalg{C}$ need be neither open nor closed as a subset of $\stab{C}$. The next two results show that whether or not it is locally closed is closely related to the structure  of the bounded closed intervals in $P(\stabalg{C})$.
\begin{lemma}
\label{locally closed lemma}
The first of the statements below implies the second and third, which are equivalent. When $\stabalg{C}$ is locally-finite all three are equivalent.
\begin{enumerate}
\item The subset $\stabalg{C}$ is \defn{locally closed} as a subspace of $\stab{C}$.
\item The inclusion $\stabalg{C} \cap \overline{S_\ts{D}} \hookrightarrow \overline{S_\ts{D}}$ is open for each algebraic $\ts{D}$.
\item For each pair of strata $S_{\ts{E},J} \tsleq \overline{S_{\ts{D},I}}$ there is an isomorphism \[[S_{\ts{E},J},S_{\ts{D},I}] \cong [I,K]^\text{op},\] where $K$ is uniquely determined by the requirement that $S_{\ts{E},J} \tsleq \partial_KS_{\ts{D},I}$.
\end{enumerate}
  \end{lemma}
\begin{proof}
Suppose $\stabalg{C}$ is locally closed. Let $\sigma \in \stabalg{C}\cap \overline{S_\ts{D}}$ where $\ts{D}$ is algebraic. Then there is a neighbourhood $U$ of $\sigma$ in $\stab{C}$ such that $U\cap \stabalg{C}$ is closed in $U$. Then $U \cap S_\ts{D} \tsleq U \cap \stabalg{C}$ so
\[
U \cap \overline{S_\ts{D}} \tsleq U \cap \stabalg{C}
\]
and $\stabalg{C} \cap \overline{S_\ts{D}}$ is open in $\overline{S_\ts{D}}$.

Now suppose $\stabalg{C} \cap \overline{S_\ts{D}}$ is open in $\overline{S_\ts{D}}$. Then we can choose a neighbourhood $U$ of $\sigma$ so that $U \cap \partial_LS_{\ts{D},I}$ is non-empty and connected for each $I \subseteq L \subseteq K$ and, moreover, $U \cap \overline{ S_{\ts{D}} } \tsleq \stabalg{C}$. It follows, as in the proof of Lemma~\ref{closed interval lemma}, that $[S_{\ts{E},J},S_{\ts{D},I}] \cong [I,K]^\text{op}$.

Conversely, if $[S_{\ts{E},J},S_{\ts{D},I}] \cong [I,K]^\text{op}$ then given a neighbourhood $U$ with $U \cap \partial_LS_{\ts{D},I}$ non-empty and connected for each $I \subseteq L \subseteq K$ we see that it meets only components of the $\partial_LS_{\ts{D},I}$ which are in $\stabalg{C}$. Hence $\stabalg{C} \cap \overline{S_\ts{D}}$ is open in $\overline{S_\ts{D}}$.

Finally, assume the stratification of $\stabalg{C}$ is locally-finite and that $\stabalg{C} \cap \overline{S_\ts{D}} \hookrightarrow \overline{S_\ts{D}}$ is open for each algebraic $\ts{D}$. Fix $\sigma\in \stabalg{C}$. There are finitely many algebraic $\ts{D}$ with $\sigma \in \overline{S_\ts{D}}$. There is an open neighbourhood $U$ of $\sigma$ in $\stab{C}$ such that
\[
U\cap \overline{S_\ts{D}} \tsleq \overline{S_\ts{D}} \cap \stabalg{C}
\]
for any algebraic $\ts{D}$ (the left-hand side is empty for all but finitely many such). Hence
\[
U \cap \stabalg{C} =U\cap \bigcup_{\ts{D} \ \text{alg}} S_\ts{D}
\tsleq U\cap \bigcup_{\ts{D} \ \text{alg}} \overline{S_\ts{D}}
= \bigcup_{\ts{D} \ \text{alg}} U\cap\overline{S_\ts{D}}
\tsleq U \cap \stabalg{C}
\]
and so $U \cap \stabalg{C} = \bigcup_{\ts{D} \ \text{alg}} U\cap\overline{S_\ts{D}}$. The latter is a {\em finite} union of closed subsets of $U$, hence closed in $U$. Therefore each $\sigma \in \stabalg{C}$ has an open neighbourhood $U \ni \sigma$ such that $U \cap \stabalg{C}$ is closed in $U$. It follows that $\stabalg{C}$ is locally closed.
\end{proof}

\begin{corollary}
Suppose $\stabalg{C}$ is locally closed as a subspace of $\stab{\cat{C}}$. Then $P(\stabalg{C})$ is pure of length $n=\rk (K\cat{C})$.
\end{corollary}
\begin{proof}
The stratum $S_{\ts{D},I}$ contains $S_{\ts{D},\{s_1,\ldots,s_n\}}$ in its closure, and is in the closure of $S_{\ts{D},\emptyset}$. It follows that any maximal chain in $P\left( \stabalg{C} \right)$ is in a  closed interval of the form $[ S_{\ts{D},\{s_1,\ldots,s_n\}}, S_{\ts{E},\emptyset} ]$. As $\stab{C}$ is locally closed this is isomorphic to the poset of subsets of an $n$-element set by Lemma~\ref{locally closed lemma}. This implies $P(\stabalg{C})$ is pure of length $n$.
\end{proof}

\begin{example}
\label{P1 three}
Recall Examples~\ref{P1 one} and~\ref{P1 two}. The subspace $\Stab_\text{alg}(\mathbb{P}^1)$ is not locally closed: if it were then $\Stab(\mathbb{P}^1) - \Stab_\text{alg}(\mathbb{P}^1) = A \cup U$ for some closed $A$ and open $U$. This subset consists of those stability conditions for which the phases of semistable objects accumulate at $\Z\subseteq \R$, and this has empty interior. Hence the only possibility is that $U=\emptyset$, in which case $\Stab_\text{alg}(\mathbb{P}^1)$ would be open. This is not the case, so $\Stab_\text{alg}(\mathbb{P}^1)$ cannot be  locally closed. Nevertheless, from the explicit description of stability conditions in \cite{MR2219846} one can see that the poset of strata is pure (of rank $2$), and that the second two conditions of Lemma~\ref{locally closed lemma} are satisfied.
\end{example}

\section{Finite-type components}\label{ftt cpts}
\subsection{The main theorem}
We say a t-structure is of \defn{finite tilting type} if it is algebraic and has only finitely many torsion-structures in its heart. A t-structure has finite tilting type if and only if it is algebraic and the interval $\itilt{\ts{D}}{\ts{D}[-1]}$ in $\tilt{C}$ is finite. We say a component $\tilto{C}$ is of \defn{finite tilting type} if each t-structure in it has finite tilting type. It follows from Lemmas~\ref{lattice} and~\ref{finite intervals} that a finite tilting type component $\tilto{C}$  is a lattice, and that closed bounded intervals in it are finite.
\begin{lemma}
\label{simple finite tilting}
Suppose that the set $S$ of t-structures obtained from some $\ts{D}$ by finite sequences of simple tilts consists entirely of t-structures of finite tilting type. Then $S$ is (the underlying set of) a finite tilting type component of $\tilt{C}$. Moreover, every finite tilting type component arises in this way.
\end{lemma}
\begin{proof}
If $\ts{D}$ has finite tilting type then any tilt of $\ts{D}$ can be decomposed into a finite sequence of simple tilts. It follows that $S$ is a component of $\tilt{C}$ as claimed. It is clearly of finite tilting type. Conversely if $\tilto{C}$ is a finite tilting type component, and $\ts{D} \in \tilto{C}$, then every t-structure obtained from $\ts{D}$ by a finite sequence of simple tilts is algebraic, and has finite tilting type. Hence $\ts{D}$ contains the set $S$, and by the first part $S=\tilto{C}$.
\end{proof}

If the heart of a t-structure contains only finitely many isomorphism classes of indecomposable objects, then it is of finite tilting type (because a torsion theory is determined by the indecomposable objects it contains). Therefore, whilst we do not use it in this paper,  the following result may be useful in detecting finite tilting type components, particularly if up to automorphism there are only finitely many t-structures which can be reached from $\ts{D}$ by finite sequences of simple tilts. In very good cases ---  for instance when tilting at a $2$-spherical simple object $s$ with the property that $\GrMor{\cat{C}}{i}{s}{s'}=0$ for $i\neq 1$ for any other simple object $s'$ ---  the tilted t-structure itself is obtained by applying an automorphism of $\cat{C}$ and hence inherits the property of being algebraic of finite tilting type. A similar situation arises if $\ts{D}$ is an algebraic t-structure in which all simple objects are rigid, \ie have no self extensions. In this case \cite[Proposition 5.4]{king-qiu} states that all simple tilts of $\ts{D}$ are also algebraic.

\begin{lemma}
\label{finite-type implies simple tilt length}
Suppose that $\ts{D}$ is a t-structure on a triangulated category $\cat{C}$ whose heart is a length category with only finitely many isomorphism classes of indecomposable objects. Then any simple tilt of $\ts{D}$ is algebraic.
\end{lemma}
\begin{proof}
It suffices to prove that the claim holds for any simple right tilt, since the simple left tilts are shifts of these. Since there are only finitely many indecomposable objects in $\heart{D}$ there are in particular only finitely many simple objects. Let these be $s_1, \ldots, s_n$ and consider the right tilt at $s_1$. Let $\sigma \in S_\ts{D}$ be the unique stability condition with $\charge_\sigma(s_1)=i$ and $\charge_\sigma(s_j)=-1$ for $j=2,\ldots,n$. Let $\tau$ be obtained by acting on $\sigma$ by $-1/2 \in \C$. Then $\ts{D}_\tau$ is the right tilt of $\ts{D}_\sigma$ at $s_1$. As there are only finitely many indecomposable objects in $\heart{D}$ the set of $\varphi\in \R$ such that $\slicing_\sigma(\varphi)\neq \emptyset$ is discrete. The same is therefore true for $\tau$. It follows that $\slicing_\tau(0,\epsilon)=\emptyset$ for some $\epsilon >0$. The component of $\stab{C}$ containing $\sigma$ and $\tau$ is full since $\sigma$ is algebraic. Hence by Lemma~\ref{alg phase criterion} the stability condition $\tau$ is algebraic too.
\end{proof}

\begin{lemma}
\label{finite-type cpts correspondence}
Let $\tilto{C}$ be a finite tilting type component of $\tilt{C}$. Then
\begin{equation}
\label{ft cpt}
\stabo{C} = \bigcup_{\ts{D} \in \tilto{C}} S_\ts{D}
\end{equation}
 is a component of $\stab{C}$.
\end{lemma}
\begin{proof}
Clearly $\tilto{C}$ is also a component of $\tiltalg{C}$. By Corollary~\ref{poset-iso} there is a corresponding component $\stabalgo{C}$ of $\stabalg{C}$ given by the RHS of (\ref{ft cpt}). Let $\stabo{C}$ be the unique component of $\stab{C}$ containing $\stabalgo{C}$. Recall from \cite[Corollary 5.2]{MR3007660} that the t-structures associated to stability conditions in  a component of $\stab{C}$ are related by finite sequences of tilts. Thus, each stability condition in $\stabo{C}$ has associated t-structure in $\tilto{C}$. In particular, the t-structure is algebraic and $\stabalgo{C}=\stabo{C}$ is actually a component of $\stab{C}$.
\end{proof}
A \defn{finite-type} component $\stabo{C}$ of $\stab{C}$ is one which arises in this way from a finite tilting type component $\tilto{C}$ of $\tilt{C}$.
\begin{lemma}
\label{lf implies lf and cf}
Suppose $\stabo{C}$ is a finite-type component. The stratification of $\stabo{C}$ is locally-finite and closure-finite.
\end{lemma}
\begin{proof}
This is immediate from Lemma~\ref{local and closure finiteness} and the obvious fact that the interval $\ialg{\ts{D}_\sigma}{\ts{D}_\sigma[-1]}$ of algebraic tilts is finite when the interval $\itilt{\ts{D}_\sigma}{\ts{D}_\sigma[-1]}$ of all tilts is finite.
\end{proof}
\begin{corollary}
\label{lf htpy eq}
Suppose $\stabo{C}$ is a finite-type component. There is a homotopy equivalence $\stabo{C} \simeq BP\left(\stabo{C}\right)$, in particular $\stabo{C}$ has the homotopy-type of a CW-complex of dimension $\dim_\C \stabo{C}$.
\end{corollary}
\begin{proof}
This is immediate from Lemma~\ref{lf implies lf and cf} and Corollary~\ref{comb model for stab alg}.
\end{proof}

We now prove that finite-type components are contractible. Our approach is modelled on the proof  of the simply-connectedness of the stability spaces of representations of Dynkin quivers \cite[Theorem 4.7]{MR3281136}. The key is to show that certain `conical unions of strata' are contractible.

The \defn{open star} $S_{\ts{D},I}^*$ of a stratum $S_{\ts{D},I}$ is the union of all strata containing $S_{\ts{D},I}$ in their closure. An open star is contractible:  $S_{\ts{D},I}^* \simeq BP(S_{\ts{D},I}^*)$ by Remark~\ref{subspaces of stabalg}, and, since $P(S_{\ts{D},I}^*)$ is a poset with lower bound $S_{\ts{D},I}$, its classifying space is  contractible.
\begin{definition}
For a finite set $F$ of  t-structures in $\tilto{C}$ let the cone
\[
C(F) = \{ (\ts{E},J) \ \colon\  \ts{F} \atileq \ts{E} \atileq \lt{E}{J} \atileq \sup F \ \text{for some}\ \ts{F}\in F  \}.
\]
Let $V(F) =\bigcup_{(\ts{E},J) \in C(F)} S_{\ts{E},J}$ be the union of the corresponding strata; we call such a subspace \defn{conical}. For example, $V(\{\ts{D}\}) = S_{\ts{D},\emptyset}$. More generally, if $F = \{ \ts{D}, L_s\ts{D} \ \colon\ s\in I\}$ then $\sup F = L_I\ts{D}$ and $V(F) = S_{\ts{D},I}^*$.
\end{definition}
\begin{remark}
\label{cones are finite}
If $(\ts{E},J) \in C(F)$ then $\inf F \atileq \ts{E} \atileq \sup F$. Since $\ialg{\inf F}{\sup F}$ is finite, and there are only finitely many possible $J$ for each $\ts{E}$, it follows that $C(F)$ is a finite set. Let $c(F) = \# C(F)$ be the number of elements, which is also the number of strata in $V(F)$.
\end{remark}
Note that $V(F)$ is an open subset of $\stabo{C}$ since $S_{\ts{D},I} \tsleq V(F)$ and $S_{\ts{D},I}\tsleq \overline{S_{\ts{E},J}}$ implies
\[
\ts{F} \atileq \ts{D} \atileq \ts{E} \atileq \lt{E}{J} \atileq \lt{D}{I} \atileq \sup F
\]
for some $\ts{F}\in F$ so that $S_{\ts{E},J} \tsleq V(F)$ too. In particular $S_{\ts{D},I} \tsleq V(F)$ implies $S_{\ts{D},I}^* \tsleq V(F)$. It is also non-empty since it contains $S_{\sup F,\emptyset}$.

\begin{proposition}
\label{conical homology vanishing}
The conical subspace $V(F)$ is contractible for any finite set $F \tsleq \tilto{C}$.
\end{proposition}
\begin{proof}
Let $C=C(F)$, $c=c(F)$, and $V=V(F)$. We prove this result by induction on the number of strata $c$. When $c=1$ we have $C=\{(\sup F,\emptyset)\}$ so that $V=S_{\sup F,\emptyset}$ is contractible as claimed. Suppose the result holds for all conical subspaces with strictly fewer than $c$ strata.

Recall from Remark~\ref{subspaces of stabalg} that $V\simeq BP(V)$ so that $V$ has the homotopy-type of a CW-complex. Hence it suffices, by the Hurewicz and Whitehead Theorems, to show that $V$ is simply-connected and that the integral homology groups $H_i(V)=0$ for $i>0$. Choose $(\ts{D},I) \in C$ such that
\begin{enumerate}
\item  $\nexists\, (\ts{E},J)\in C$ with $\ts{E} \prec \ts{D}$;
\item  $(\ts{D},I')\in C \iff I'\subseteq I$.
\end{enumerate}
It is possible to choose such a $\ts{D}$ since $C$ is finite; note that $\ts{D}$ is necessarily in $F$. It is then possible to choose such an $I$ because if $S_{\ts{D},I'},S_{\ts{D},I''} \tsleq V$ then $\lt{D}{I'} ,\lt{D}{I''} \atileq \sup F$ which implies  $\lt{D}{I'\cup I''} = \lt{D}{I'} \vee \lt{D}{I''} \atileq \sup F$.

The conical subset $V$ has an open cover $V = S_{\ts{D},I}^* \cup (V-S_{\ts{D}})$. We remarked above that $S_{\ts{D},I}^*$ is contractible. In addition, by the choice of $\ts{D}$, the subspace $V-S_\ts{D} = V(F')$  is also conical, with
\[
F'= F \cup \{\lt{D}{s}  \ \colon\  s\in \ts{D}^\circ \ \text{simple}, \lt{D}{s} \atileq \sup F\} -\{\ts{D}\}.
\]
Since $V(F')$ has fewer strata than $V$ it is contractible by the inductive hypothesis. Finally, the intersection $S_{\ts{D},I}^* \cap (V-S_{\ts{D}}) = S_{\ts{D},I}^* -S_{\ts{D}}$ is the conical subspace
\[
 \bigcup_{\ts{D} \prec \ts{E} \atileq \lt{E}{J} \atileq \lt{D}{I}} S_{\ts{E},J}= V\left(  \{ \lt{D}{s} \ \colon\  s\in I\}  \right),
\]
which has fewer strata than $V$. Hence this too is contractible by the inductive hypothesis. It follows that $V$ is simply-connected by the van Kampen Theorem, and that $H_i(V)=0$ for $i>0$ by the Mayer--Vietoris sequence for the open cover by $S_{\ts{D},I}^*$ and $V-S_{\ts{D}}$. Hence $V$ is contractible by the Hurewicz and Whitehead Theorems. This completes the inductive step.
\end{proof}

\begin{theorem}
\label{stab contractible}
Suppose $\stabo{C}$ is a finite-type component. Then $\stabo{C}$ is contractible.
\end{theorem}
\begin{proof}
By Lemma~\ref{lf implies lf and cf} $\stabo{C}$ is a locally-finite stratified space. Thus a singular integral $i$-cycle in $\stabo{C}$ has support meeting only finitely many strata, say the support is contained in $\{S_\ts{F} \ \colon\  \ts{F}\in F\}$. Therefore the cycle  has support in $V(F)$, and so is null-homologous whenever $i>0$ by Proposition~\ref{conical homology vanishing}. This shows that $H_i(\stabo{C})=0$ for $i>0$. An analogous argument shows that $\stabo{C}$ is simply-connected. Since $\stabo{C}$ has the homotopy type of a CW-complex it follows from the Hurewicz and Whitehead Theorems that $\stabo{C}$ is contractible.
\end{proof}

We discuss two classes of examples of triangulated categories in which each component of the stability space is of finite-type, and hence is contractible. Each class contains the bounded derived category of finite-dimensional representations of ADE Dynkin quivers, so these can be seen as two ways to generalise from these.

\subsection{Locally-finite triangulated categories}
\label{locally-finite categories}
We recall the definition of locally-finite triangulated category from \cite{MR2955969}. Let $\cat{C}$ be a triangulated category. The \defn{abelianisation} $\ab(\cat{C})$ of $\cat{C}$ is the full subcategory of functors $F \colon \cat{C}^\text{op} \to \ab$ 
fitting into an exact sequence
\[
\Mor{\cat{C}}{-}{c} \to \Mor{\cat{C}}{-}{c'} \to F \to 0
\]
for some $c,c'\in \cat{C}$. The Yoneda embedding $\cat{C} \to \ab(\cat{C})$ is the universal cohomological functor on $\cat{C}$, in the sense that any cohomological functor to an abelian category factors, essentially uniquely, as the Yoneda embedding followed by an exact functor.
\begin{defn}
A triangulated category\footnote{Our default assumption that all categories are essentially small is necessary here.} $\cat{C}$ is \defn{locally-finite} if idempotents split and its abelianisation $\ab(\cat{C})$ is a length category.
\end{defn}
The following `internal' characterisation is due to Auslander \cite[Theorem 2.12]{MR0349747}.
\begin{proposition}
A triangulated category $\cat{C}$ in which idempotents are split is locally-finite if and only if for each $c\in \cat{C}$
\begin{enumerate}
\item there are only finitely many isomorphism classes of indecomposable objects $c'\in \cat{C}$ with $\Mor{\cat{C}}{c'}{c} \neq 0$;
\item for each indecomposable $c'\in \cat{C}$, the $\mathrm{End}_\cat{C}(c')$-module $\Mor{\cat{C}}{c'}{c}$ has finite length.
\end{enumerate}
\end{proposition}
The category $\cat{C}$ is locally-finite if and only if $\cat{C}^\text{op}$ is locally-finite so that the above properties are equivalent to the dual ones.

Locally-finite triangulated categories have many good properties:  they have a Serre functor,
equivalently by \cite{RV} they have Auslander--Reiten triangles, the inclusion of any thick subcategory has both left and right adjoints, any thick subcategory, or  quotient thereby, is also locally-finite. See \cite{MR2955969,MR2430189,MR2153264} for further details.

\begin{lemma}[{\cf \cite[Proposition 6.1]{pauk}}]
\label{alg implies finite-type}
Suppose that $\cat{C}$ is a locally-finite triangulated category $\cat{C}$ with $\rk K\cat{C}<\infty$. Then any t-structure on $\cat{C}$ is algebraic, with only finitely many isomorphism classes of indecomposable objects in its heart.
\end{lemma}
\begin{proof}
Let $d$ be an object in the heart of a t-structure, and suppose it has infinitely many pairwise non-isomorphic subobjects. Write each of these as a direct sum of the indecomposable objects with non-zero morphisms to $d$. Since there are only finitely many isomorphism classes of such indecomposable objects, there must be one of them, $c$ say, such that $c^{\oplus k}$ appears in these decompositions for each $k=1,2,\ldots$. Hence $c^{\oplus k} \hookrightarrow d$ for each $k$, which contradicts the fact that $\Mor{\cat{C}}{c}{d}$ has finite length as an $\mathrm{End}_\cat{C}(c)$-module (because it has a filtration by $\{ \alpha \colon c\to d \ \colon\ \alpha\ \text{factors through}\ c^{\oplus k} \to d\}$ for $k\in \N$). We conclude that any object in the heart has only finitely many pairwise non-isomorphic subobjects. It follows that the heart is a length category. Since $\rk K\cat{C}<\infty$ it has finitely many simple objects, and so is algebraic.

To see that there are only finitely many indecomposable objects (up to isomorphism) note that any indecomposable object in the heart has a simple quotient. There are only finitely many such simple objects, and each of these admits non-zero morphisms from only finitely many isomorphism classes of indecomposable objects.
\end{proof}

\begin{remark}
\label{finite torsion poset}
Since a torsion theory is determined by its indecomposable objects it follows that a t-structure on $\cat{C}$ as above has only finitely many torsion structures on its heart, \ie it has finite tilting type.
\end{remark}

\begin{corollary}
\label{LF contractibility}
Suppose $\cat{C}$ is a locally-finite triangulated category and that  $\rk K\cat{C}< \infty$. Then the stability space is a (possibly empty) disjoint union of finite-type components, each of which is contractible.
\end{corollary}
\begin{proof}
Combining Lemma~\ref{alg implies finite-type} with Lemma~\ref{simple finite tilting} shows that each component of the tilting poset is of finite tilting type. The result follows from Theorem~\ref{stab contractible}.
\end{proof}
\begin{example}
\label{Dynkin quiver}
 Let $Q$ be a quiver whose underlying graph is an ADE Dynkin diagram, and suppose the field $k$ is algebraically-closed. Then $\cat{D}(Q)$ is a locally-finite triangulated category \cite[\S2]{MR3050703}. The space $\Stab(Q)$ of stability conditions is non-empty and connected (by Remark~\ref{connectedness in algebra case} or the results of \cite{MR976638}), and hence by Corollary~\ref{LF contractibility} is contractible. This affirms the first part of \cite[Conjecture 5.8]{MR3281136}. Previously $\Stab(Q)$ was known to be simply-connected \cite[Theorem 4.7]{MR3281136}.
\end{example}
\begin{example}
\label{cluster category}
For $m\geq 1$ the cluster category $\cat{C}_m(Q) = \cat{D}(Q) / \Sigma_m$ is the quotient of $\cat{D}(Q)$ by the automorphism  $\Sigma_m= \tau^{-1} [m-1]$, where $\tau$ is the Auslander--Reiten translation. Each $\cat{C}_m(Q)$ is locally-finite \cite[\S2]{MR2955969}, but $\Stab(\cat{C}_m(Q)) = \emptyset$ because there are no t-structures on $\cat{C}_m(Q)$.

Remark 5.6 of \cite{MR3281136} proposes that $\stabGQ{N} / \braidGQ{N}$ should be considered as an appropriate substitute for the stability space of $C_{N-1}(Q)$. Our results show that the former is homotopy equivalent to the classifying space of the braid group $\braidGQ{N}$, which might be considered as further support for this point of view.
\end{example}

\subsection{Discrete derived categories}\label{sec:DDC}
This class of triangulated categories was introduced and classified by Vossieck \cite{MR1851659}; we use the more explicit classification in \cite{MR2041666}. The contractibility of the stability space, Corollary~\ref{discrete contractibility} below, follows from the results of this paper combined with the detailed analysis of t-structures on these categories in \cite{pauk}. \cite[Theorem 7.1]{pauk2} provides an independent proof of the contractibility of $B\PI{C}$ for a discrete derived category $\cat{C}$, using the interpretation of $\PI{C}$ in terms of the poset $\mathbb{P}_2(\cat{C})$ of silting pairs (Remark~\ref{BPP2}). Combining this with Corollary~\ref{comb model for stab alg} one obtains an alternative proof \cite[Theorem 8.10]{pauk2} of the contractibility of the stability space.

Let $A$ be a finite-dimensional associative algebra over an algebraically-closed field. Let $\cat{D}(A)$ be the bounded derived category of finite-dimensional right $A$-modules.
\begin{definition}
The derived category $\cat{D}(A)$ is \defn{discrete} if for each map (of sets) $\mu \colon \Z \to K\left(\cat{D}(A)\right)$ there are only finitely many isomorphism classes of objects $d\in \cat{D}(A)$ with $[H^id] = \mu(i)$ for all $i\in \Z$.
\end{definition}
The derived category $\cat{D}(Q)$ of a quiver whose underlying graph is an ADE Dynkin diagram is discrete. \cite[Theorem A]{MR2041666} states that if $\cat{D}(A)$ is discrete but not of this type then it is equivalent as a triangulated category to $\cat{D}\left( \Lambda(r,n,m) \right)$ for some $n\geq r \geq 1$ and $m\geq 0$ where $\Lambda(r,n,m)$ is the path algebra of the bound quiver in Figure~\ref{discrete quiver}. Indeed, $\cat{D}(A)$ is discrete if and only if $A$ is tilting-cotilting equivalent either to the path algebra of an ADE Dynkin quiver or to one of the $\Lambda(r,n,m)$.
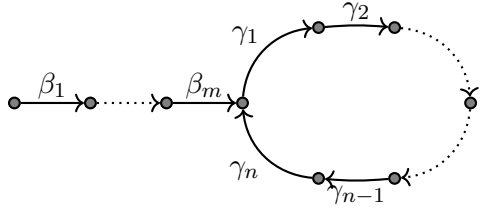
\begin{figure}[htb]
\begin{center}
\begin{tikzpicture}
[auto, inner sep=0.5mm, vertex/.style={circle,draw,fill=gray,thick}];

\node (tail1) at (0,0) [vertex] {};
\node (tail2) at (1,0) [vertex] {};
\node (tail3) at (2,0) [vertex] {};
\node (join) at (3,0) [vertex] {};

\draw[->,thick] (tail1) to node {$\beta_1$} (tail2);
\draw[->,dotted,thick] (tail2) to (tail3);
\draw[->,thick] (tail3) to node {$\beta_m$} (join);

\node (cycle1) at (4,1) [vertex]{};
\node (cycle2) at (5,1) [vertex]{};
\node (cycle3) at (6,0) [vertex]{};
\node (cycle4) at (5,-1) [vertex]{};
\node (cycle5) at (4,-1) [vertex]{};

\draw[->,thick] (join) to [bend left=40] node {$\gamma_1$} (cycle1);
\draw[->,thick] (cycle1) to [bend left=5] node{$\gamma_2$} (cycle2);
\draw[->,dotted,thick] (cycle2) to [bend left=40] (cycle3);
\draw[->,dotted,thick] (cycle3) to [bend left=40] (cycle4);
\draw[->,thick] (cycle4) to [bend left=5] node{$\gamma_{n-1}$} (cycle5);
\draw[->,thick] (cycle5) to [bend left=40] node{$\gamma_{n}$} (join);
\end{tikzpicture}
\end{center}
\caption{The algebra $\Lambda(r,n,m)$ is the path algebra of the quiver $Q(r,n,m)$ above with relations $\gamma_{n-r+1}\gamma_{n-r+2}=\cdots = \gamma_{n}\gamma_1=0$. }
\label{discrete quiver}
\end{figure}

Discrete derived categories form an interesting class of examples as they are intermediate between the locally-finite case considered in the previous section and  derived categories of tame representation type algebras. More precisely, the distinctions are captured by the Krull--Gabriel dimension of the abelianisation, which measures how far the latter is from being a length category. In particular, $\KGdim \left(\ab(\cat{C})\right) \leq 0$ if and only if $\cat{C}$ is locally-finite \cite{krause}. Krause conjectures \cite[Conjecture 4.8]{krause} that $\KGdim \left(\ab\left(\cat{D}(A) \right) \right)= 0$ or $1$ if and only if $\cat{D}(A)$ is discrete. As evidence he shows that for  the full subcategory $\mathrm{proj}\, \mathbf{k}[\epsilon]$ of finitely generated projective modules over the algebra  $\mathbf{k}[\epsilon]$ of dual numbers, $\KGdim \left( \ab\left(\cat{D}_b(\mathrm{proj}\, \mathbf{k}[\epsilon]) \right) \right)= 1$. The bounded derived category $\cat{D}(\mathrm{proj}\, \mathbf{k}[\epsilon])$ is discrete ---  there are infinitely many indecomposable objects, even up to shift, but no continuous families ---  but not locally-finite. Finally, by \cite[Theorem 4.3]{MR808682} $\KGdim\left(\cat{D}(A)\right)=2$ when $A$ is a tame hereditary Artin algebra, for example the path algebra of the Kronecker quiver $\widetilde{A_1}$.

Since the Dynkin case was covered in the previous section we restrict to the categories $\cat{D}\left( \Lambda(r,n,m) \right)$. These have finite global dimension if and only if  $r<n$, and we further restrict to this situation.
\begin{corollary}[\cf {\cite[Theorem 8.10]{pauk2}}]
\label{discrete contractibility}
Suppose $\cat{C} = \cat{D}\left( \Lambda(r,n,m) \right)$, where $n>r \geq 1$ and $m\geq 0$. Then the stability space $\stab{C}$  is  contractible.
\end{corollary}
\begin{proof}
By \cite[Proposition 6.1]{pauk} any t-structure on $\cat{C}$ is algebraic with only finitely many isomorphism classes of indecomposable objects in its heart. Lemma~\ref{simple finite tilting} then shows that each component of the tilting poset has finite-type. By Theorem~\ref{stab contractible} $\stab{C}=\stabalg{C}$, and is a union of contractible components. By Lemma~\ref{stabalg connected} $\stabalg{C}$ is connected. Hence $\stab{C}$ is contractible.
\end{proof}
\begin{example}
The space of stability conditions in the simplest case, $(n,r,m)=(2,1,0)$, was computed in \cite{MR2739061} and shown to be $\C^2$. (The category was described geometrically in \cite{MR2739061}, as the constructible derived category of $\P^1$ stratified by a point and its complement, but it is known that in this case the constructible derived category is equivalent to the derived category of the perverse sheaves, and these have a nearby and vanishing-cycle description as representations of the quiver $Q(2,1,0)$ with relation $\gamma_2\gamma_1=0$.)
\end{example}

\section{The {C}alabi-{Y}au-{$N$}-category of a {D}ynkin quiver}
\label{CYNGinzburg}
\subsection{The category}
In this section we consider in detail another important example of a finite-type component, associated to the Ginzburg algebra of an ADE Dynkin quiver. We also address the related question of the faithfulness of the braid group action on the associated derived category.

Let $Q$ be a quiver whose underlying unoriented graph is an ADE Dynkin diagram. Fix $N\geq 2$ and let $\GQ{N}$ be the associated Ginzburg algebra of degree $N$, let $\catGQ{N}$ be the bounded derived category of finite-dimensional representations of $\GQ{N}$ over an algebraically-closed field $\mathbf{k}$, and let $\Stab(\GQ{N})$ be the space of stability conditions on $\catGQ{N}$. See \cite[\S7]{MR3050703} for the details of the construction of the differential-graded algebra $\GQ{N}$ and its derived category, and for a proof that $\catGQ{N}$ is a Calabi--Yau-$N$ category. (Recall that a $k$-linear triangulated category $\cat{C}$ is \defn{Calabi--Yau-$N$}
if, for any objects $c,c'$ in $\cat{C}$ we have a natural isomorphism
\begin{equation}
\label{eq:serre}
    \mathfrak{S} \colon \GrMor{\cat{C}}{\bullet}{c}{c'}
        \xrightarrow{\sim}\GrMor{\cat{C}}{\bullet}{c'}{c}^\vee[N].
\end{equation}
Here the graded dual of a graded vector space
$V=\oplus_{i\in\Z} V_i[i]$  is defined by $V^\vee=\oplus_{i\in\Z} V_i^*[-i]$.)
By \cite{AMY}, $\tiltGQ{N}$ and $\Stab(\GQ{N})$ are connected.

\begin{corollary}
\label{CYN contractibility}
The stability space $\stabGQ{N}$ is of finite-type, and hence is contractible.
\end{corollary}
\begin{proof}
By \cite[Corollary 8.4]{king-qiu} each t-structure obtained from the standard one, whose heart is the representations of $\GQ{N}$, by a finite sequence of simple tilts is algebraic. \cite[Lemma 5.1 and Proposition 5.2]{MR3281136} show that each of these t-structures is of finite tilting type.
Hence by Lemma~\ref{simple finite tilting} $\tiltGQ{N}$ has finite tilting type,
and therefore by Theorem~\ref{stab contractible} $\stabGQ{N}$ is contractible.
\end{proof}
This affirms the second part of \cite[Conjecture 5.8]{MR3281136}.

\subsection{The braid group}
\label{spherical twist group}
An object $s$ of a $k$-linear triangulated category is \defn{$N$-spherical} if $\GrMor{\cat{C}}{\bullet}{s}{s}\cong \mathbf{k} \oplus \mathbf{k}[-N]$ and \eqref{eq:serre} holds functorially for $c=s$ and any $c'$ in $\cat{C}$. The \defn{twist functor $\varphi_s$ of a spherical object}   $s$ was  defined in \cite{MR1831820} to be
\begin{equation}
\label{eq:phi}
    \varphi_s(c)=\Cone\left(s\otimes\GrMor{}{\bullet}{s}{c}\to c\right)
\end{equation}
with inverse $\varphi_s^{-1}(c)=\Cone\left(c\to s\otimes\GrMor{}{\bullet}{s}{c}^\vee \right)[-1]$.
Denote by $\GQzero$ the canonical heart in $\catGQ{N}$, which is equivalent to the module category of $Q$.  Each simple object in $\GQzero$ is $N$-spherical
cf. \cite[\S~7.1]{king-qiu}.
The \defn{braid group} or \defn{spherical twist group} $\braidGQ{N}$ of $\catGQ{N}$ is the subgroup of $\Aut\catGQ{N}$ generated by $\{\varphi_{s}\ \colon\  s \text{ is simple in }\GQzero\}$. The lemma below follows directly from the definition of spherical twists.
\begin{lemma}
\label{twists and functors}
Let $\cat{C}$ be a $k$-linear triangulated category, $\varphi_s$ a spherical twist, and $F$ any auto-equivalence. Then $F\circ \varphi_s = \varphi_{F(s)}\circ F$.
\end{lemma}
An important consequence is that two twists $\varphi_s$ and $\varphi_t$ by simple objects $s$ and $t$ satisfy the
\begin{itemize}
\item braid relation $\varphi_s\varphi_t\varphi_s = \varphi_t\varphi_s\varphi_t$
if and only if $\GrMor{}{\bullet}{s}{t}\cong k[-j]$ for some $j\in\mathbb{Z}$;
\item commutativion relation $\varphi_s\varphi_t= \varphi_t\varphi_s$
if and only if $\GrMor{}{\bullet}{s}{t}=0$;
\end{itemize}
It follows that there is a surjection
\begin{equation}
\label{braid surjection}
\braidsurj_N \colon \braidQ \epic \braidGQ{N}.
\end{equation}
from the braid group $\braidQ$ of the underlying Dynkin diagram, which has a generator $b_i$ for each vertex $i$ and relations $b_ib_jb_i=b_jb_ib_j$ when there is an edge between vertices $i$ and $j$, and $b_ib_j=b_jb_i$ otherwise. We will show that $\braidsurj_N$ is an isomorphism for any $N\geq 2$. We deal with the cases when $N=2$, and when $Q$ has type $A$ (for any $N\geq 2$) below; these are already known but we obtain new proofs.

Let $\g$ be the finite-dimensional complex simple Lie algebra associated to the underlying Dynkin diagram of $Q$. Let $\h\subseteq \g$ denote the Cartan subalgebra and let $\hreg\subseteq \h$ be the complement of the root hyperplanes in $\h$, \ie
\[\hreg=\{\theta\in \h:\theta(\alpha)\neq 0\text{ for all }\alpha\in \Lambda\},\]
where $\Lambda$ is a set of simple roots, i.e.\ a basis of $\h$ such that each root can be written as an integral linear combination of basis vectors with either all non-negative or all non-positive coefficients.
The Weyl group $W$ is generated by reflections in the root hyperplanes and acts freely on $\hreg$.

\begin{theorem}[{\cite[Theorem~1.1]{MR2549952}}]\label{thm:Br}
Let $Q$ be an ADE Dynkin quiver. Then $\stabGQ{2}$ is a covering space of $\hreg/W$
and $\Br(\Gamma_2Q)$  preserves this component and acts as the group of deck transformations.
\end{theorem}

It is well-known that the fundamental group of $\hreg/W$ is the braid group $\braidQ$ associated to the quiver $Q$. We therefore obtain new proofs for the following two theorems,
by combining Theorem~\ref{thm:Br} and Corollary~\ref{CYN contractibility}.
\begin{theorem}[{\cite[Theorem~1.1]{MR2854121}}]\label{thm:BT}
Let $Q$ be an ADE Dynkin quiver. Then $\Phi_2 \colon \braidQ \to \Br(\Gamma_2Q)$ is an isomorphism.

\end{theorem}
\begin{theorem}[{\cite{MR0422673}}]\label{thm:Deligne}
The universal cover of $\hreg/W$ is contractible.
\end{theorem}

Ikeda has extended Bridgeland--Smith's work relating stability conditions with quadratic differentials to obtain the following result.
\begin{theorem}[{\cite[Theorem~1.1]{Akishi}}]\label{thm:AI}
Let $Q$ be a Dynkin quiver of type $A$. Then there is an isomorphism
$\stabGQ{N}/\braidGQ{N}\cong\hreg/W$ of complex manifolds.
\end{theorem}
Combining this with Corollary~\ref{CYN contractibility}, we obtain a new proof of

\begin{theorem}[{\cite{MR1831820}}]\label{thm:ST}
Let $Q$ be a quiver of type $A$. Then  $\Phi_N \colon \braidQ\to \braidGQ{N}$ is an isomorphism.
\end{theorem}

Unfortunately we do not yet know enough about the geometry of the stability spaces for the Calabi--Yau-$N$ categories constructed from Dynkin quivers of other types to deduce  the analogous faithfulness of the braid group in those cases. In \S\ref{braid action free} we give an alternative proof of faithfulness which works for all Dynkin quivers (Corollary~\ref{braid action free 2}), which also provides a new proof of Theorem~\ref{thm:Deligne}.

Although not phrased in these terms, the above proof is equivalent to showing that the action of $\braidQ$ on the combinatorial model $\mathrm{Int}^\circ\!\left(\catGQ{N}\right)$ of $\stabGQ{N}$ is free. The alternative proof in \S\ref{braid action free} proceeds by showing instead that the action of $\braidQ$ on $\tiltGQ{N}$ is free.

\section{The braid action is free}\label{braid action free}

In this section we show that the action of the braid group on $\tiltGQ{N}$ via the surjection $\braidsurj_N \colon \braidQ \to \braidGQ{N}$ is free. Our strategy uses the isomorphism $\braidsurj_2 \colon \braidQ \to \braidGQ{2}$ from Theorem~\ref{thm:AI} as a key step, i.e.\ we bootstrap from the $N=2$ case. Therefore we assume $N\geq 3$ unless otherwise specified.

For ease of reading we will usually omit $\braidsurj_N$ from our notation when discussing the action, writing simply $b\cdot \ts{D}$ for $\braidsurj_N(b) \ts{D}$ where $b\in \braidQ$ and $\ts{D}\in \tiltGQ{N}$.

\subsection{Local Structure of $\tiltGQ{N}$}

We describe the intervals from $\ts{D}$ to $L_{\langle s_i,s_j\rangle}\ts{D}$ where $s_i$ and $s_j$ are distinct simple objects of the heart of some $\ts{D}$. It will be convenient to consider $\tiltGQ{N}$ as a category, with objects the elements of the poset and with a unique morphism $\ts{D}\to \ts{E}$ whenever $\ts{D}\leq \ts{E}$. The following lemma is the analogue for $\catGQ{N}$ of \cite[Lemma~4.3]{MR3281136}.
\begin{lemma}
\label{lem:45}
Suppose $s_i$ and $s_j$ are distinct simple objects of the heart of a t-structure $\ts{D}\in\tiltGQ{N}$. Then there is either a square or pentagonal commutative diagram of the form
\begin{gather}\label{45}
\xymatrix@C=1.4pc@R=1pc{
    &  L_{s_i}\ts{D}    \ar[dr]^{}\\
    \ts{D}  \ar[ur]^{} \ar[dr]_{} &&  L_{\langle s_i, s_j\rangle}\ts{D}\\
    &   L_{s_j}\ts{D}  \ar[ur]_{}
}\qquad
\xymatrix@C=1pc@R=0.8pc{
    &   L_{s_i}\ts{D}  \ar[rr]^{} &&      \ts{D}' \ar[dd]^{} \\
   \ts{D}\ar[ur]^{}\ar[dr]_{}     \\
    &   L_{s_j}\ts{D}   \ar[rr]_{}  &&L_{\langle s_i, s_j\rangle}\ts{D}
}
\end{gather}
in $\tiltGQ{N}$, where we may need to exchange $i$ and $j$ to get the precise diagram in the pentagonal case, and the t-structure $\ts{D}'$ is uniquely specified by the diagram. The square occurs when $\GrMor{}{1}{s_i}{s_j} = 0 = \GrMor{}{1}{s_j}{s_i}$ and the pentagon occurs when $\GrMor{}{1}{s_i}{s_j} = 0$ and  $\GrMor{}{1}{s_j}{s_i} \cong k$.
\end{lemma}
\begin{proof}
First, we claim that either $\GrMor{}{1}{s_i}{s_j} = 0 = \GrMor{}{1}{s_j}{s_i}$ or that  $\GrMor{}{1}{s_i}{s_j} = 0$ and  $\GrMor{}{1}{s_j}{s_i} \cong k$. Let the set of simple objects in the heart of $\ts{D}$ be $\{s_1,\ldots,s_n\}$. By \cite[Corollary~8.4 and Proposition~7.4]{king-qiu}, there is a t-structure $\ts{E}$ in $\cat{D}(Q)$ such that the Ext-quiver of the heart of $\ts{D}$ is the Calabi--Yau-$N$ double of the Ext-quiver of the heart of $\ts{E}$. In other words, one can label the simple objects in the latter as $\{t_1,\ldots,t_n\}$ in such a way that
\begin{gather}\label{eq:CY}
    \dim\GrMor{}{d}{s_k}{s_l}= \dim\GrMor{}{d}{t_k}{t_l} + \dim\GrMor{}{N-d}{t_l}{t_k}
\end{gather}
for any $1\leq k,l\leq n$. Moreover, by \cite[Lemma~4.2]{MR3281136}, we have
\[
   \dim\GrMor{}{\bullet}{t_k}{t_l} + \dim\GrMor{}{\bullet}{t_l}{t_k}\leq 1,
\]
for any $1\leq k,l\leq n$. So we may assume, without loss of generality, that $\GrMor{}{\bullet}{t_i}{t_j} =0$
and $\GrMor{}{\bullet}{t_j}{t_i}$ is either zero or is one-dimensional and concentrated in degree $d$ for some $d\in\mathbb{Z}$. Therefore, as $N\geq 3$,
\begin{align*}
    \dim \GrMor{}{1}{s_i}{s_j} +&\dim \GrMor{}{1}{s_j}{s_i} =\\
    &\dim\GrMor{}{N-1}{t_j}{t_i}+\dim\GrMor{}{1}{t_j}{t_i}  \leq1
\end{align*}
and the claim follows. Since the simple objects $\{s_1,\ldots,s_n\}$ are $N$-spherical, and $N\geq 3$, we also note that $\GrMor{}{1}{s_i}{s_i} = 0 = \GrMor{}{1}{s_j}{s_j}$ so that neither $s_i$ nor $s_j$ has any self-extensions.

The required diagrams arise from the poset of torsion theories in the heart of $\ts{D}$ which are contained in the extension-closure $\langle s_i,s_j\rangle$. This is the same as the poset of torsion theories in the full subcategory $\langle s_i,s_j\rangle$. When $\GrMor{}{1}{s_i}{s_j} = 0 = \GrMor{}{1}{s_j}{s_i}$ this subcategory is equivalent to representations of the quiver with two vertices and no arrows, and when $\GrMor{}{1}{s_j}{s_i} = 0$ and  $\GrMor{}{1}{s_i}{s_j} \cong \mathbf{k}$ it is equivalent to representations of the $A_2$ quiver. Identifying torsion theories with the set of non-zero indecomposable objects contained within them we have four in the first case ---  $\emptyset$, $\{s_j\}$, $\{s_i\}$, and $\{s_j,s_i\}$ ---  and five in the second ---  $\emptyset$, $\{s_j\}$, $\{s_i\}$, $\{e,s_i\}$, and $\{s_j,s_i\}$ where $e$ is the indecomposable extension $0\to s_j \to e \to s_i \to 0$. These clearly give rise to the square and pentagonal diagrams above. Moreover, note that $\ts{D}' = L_{\langle s_i, e\rangle}\ts{D}$ is uniquely specified as claimed.
\end{proof}

\begin{remark}
\label{generators and relations for tilt}
Recall from Lemma~\ref{lattice} that $\tiltGQ{N}$ is a lattice. It follows that the above lemma allows us to give a presentation for the category $\tiltGQ{N}$ in terms of generating morphisms and relations. The generators are the simple left tilts. The relations are provided by the squares and pentagons of the above lemma.
\end{remark}

\subsection{Associating generating sets}
\label{sec:gs}

By \cite[Corollary 8.4]{king-qiu} the simple objects of the heart of any t-structure in $\tiltGQ{N}$ are $N$-spherical, and the associated spherical twists form a generating set for $\braidGQ{N}$. Moreover, we can explicitly describe how the generating set changes as we perform a simple tilt. Let $s_1,\ldots,s_n$ be the simple objects of the heart of $\ts{D}$. By \cite[Proposition~5.4 and Remark~7.1]{king-qiu}, the simple objects of the heart of $L_{s_i} \ts{D}$
are
\begin{equation}
\label{simple change 1}
    \{s_i[-1]\}\cup\{s_k\ \colon\  \GrMor{}{1}{s_i}{s_k}=0, k\neq i\}\cup\{\varphi_{s_i}(s_j)\ \colon\  \GrMor{}{1}{s_i}{s_j}\neq0\}.
\end{equation}
As $\varphi_{\varphi_{s_i}(s_j)} = \varphi_{s_i} \varphi_{s_j}\varphi^{-1}_{s_i}$ by Lemma~\ref{twists and functors},
\begin{equation}
\label{generator change 1}
    \{\varphi_{s_i}\}\cup\{\varphi_{s_k}\ \colon\  \GrMor{}{1}{s_i}{s_k}=0 \}\cup\{\varphi_{s_i}\varphi_{s_j}\varphi_{s_i}^{-1}\ \colon\  \GrMor{}{1}{s_i}{s_j}\neq0\}
\end{equation}
is the new generating set for $\braidGQ{N}$.
In this section we lift the above generating sets, in certain cases, along the surjection $\braidsurj_N$ to generating sets for $\braidQ$.

Let $\GQzero$ be the standard t-structure in $\catGQ{N}$. By \cite[Theorem~8.6]{king-qiu} there is a canonical bijection
\begin{gather}\label{eq:fund}
   \tiltGQo{N} \xrightarrow{1-1}  \tiltGQ{N}/\braidGQ{N},
\end{gather}
where $\tiltGQo{N}$ is the full subcategory of $\tiltGQ{N}$ consisting of t-structures between $\GQzero$ and $\GQzero[2-N]$. 
Let $\Qzero$ be the standard t-structure in $\cat{D}(Q)$ and let $\tiltQo{N}$ be the full subcategory of $\tiltQ$ consisting of t-structures
between $\Qzero$ and $\Qzero[2-N]$. Recall from  \cite[Definition 7.3, \S8]{king-qiu} that there is a strong Lagrangian immersion $\LI{N} \colon\cat{D}(Q)\to\catGQ{N}$, i.e.\ a triangulated functor with the additional property that
for any $x,y\in\cat{D}(Q)$,
\begin{gather}\label{eq:double}
    \GrMor{}{d}{\LI{N}(x)}{\LI{N}(y)}\cong \GrMor{}{d}{x}{y} \oplus  \GrMor{}{N-d}{y}{x}^*.
\end{gather}
In this case, by \cite[Theorem~8.6]{king-qiu}, the Lagrangian immersion induces an isomorphism
\begin{gather}\label{eq:==}
    \LI{N}_* \colon \tiltQo{N} \to \tiltGQo{N},
\end{gather}
sending $\Qzero$ to $\GQzero$. Moreover, for $\ts{E} \in\tiltQo{N}$ the simple objects of the heart of $\LI{N}_*(\ts{E})\in\tiltGQo{N}$ are the images under $\LI{N}$ of the simple objects of the heart of $\ts{E}$.

Denote by $\Ind\cat{C}$ the set of indecomposable objects in an additive category $\cat{C}$.
For any  acyclic quiver $Q$, it is known that $\Ind \cat{D}(Q)=\bigcup_{l\in\mathbb{Z}}\Ind\Qzero[l]$ where $\Qzero$ is the standard heart. By Theorem~\ref{thm:BT} there is an isomorphism $\braidsurj_2^{-1} \colon\braidGQ{2}\to\braidQ$. We define a map
\[
    b \colon \Ind\cat{D}(Q)\to \braidQ \colon x\mapsto \braidsurj_2^{-1}(\varphi_{\LI{2}(x)}).
\]
To spell it out, we first send $x$ to $\LI{2}(x)$, which is a $2$-spherical object in $\catGQ{2}$ (see the lemma below),
and then take the image of its spherical twist in $\braidQ$ under the isomorphism $\braidsurj_2^{-1}$. Note that $b$ is invariant under shifts.
\begin{lemma}\label{lem:key}
Let $x, y\in\Ind\cat{D}(Q)$. Then
\begin{enumerate}
\item\label{key:spherical} $\LI{2}(x)$ is a $2$-spherical object for any $x\in\Ind\cat{D}(Q)$;
\item\label{key:commute} if $\GrMor{}{\bullet}{x}{y} = \GrMor{}{\bullet}{y}{x} = 0$, then
$b(x) b(y)= b(y) b(x)$;
\item\label{key:braid} if there is a triangle $y\to z\to x\to y[1]$ in $\Ind\cat{D}(Q)$ for some
some $z\in\Ind\cat{D}(Q)$, then $b(z)=b(x) b(y) b(x)^{-1}$ and
\[
b(x) b(y) b(x)= b(y) b(x) b(y),
\]
i.e.\ $b(x)$ and $b(y)$ satisfy the braid relation.
\end{enumerate}
\end{lemma}
\begin{proof}
Let $x$ be an indecomposable in $\cat{D}(Q)$. Then, by \cite[Lemma~2.4]{MR3281136}, $x$ induces a section $P(x)$ of the Auslander--Reiten quiver of $\cat{D}(Q)$, and hence a t-structure $\ts{D}_x=[P(x),\infty)$. For a Dynkin quiver, all such t-structures are known to be related to the standard t-structure by tilting, so $\ts{D}_x\in \tiltQ$. Moreover, again by \cite[Lemma~2.4]{MR3281136}, the heart of $\ts{D}_x$ is isomorphic to the category of $kQ'$ modules for some quiver $Q'$ with the same underlying diagram as $Q$. It follows that the section $P(x)$ is isomorphic to $(Q')^\text{op}$ and consists of the projective representations of $kQ'$. By definition $x$ is a source of the section, so is the projective corresponding to a sink in $Q'$, and is therefore a simple object of the heart. By \cite[Corollary 8.4]{king-qiu} the image of any such simple object is $2$-spherical. Hence (\ref{key:spherical}) follows.

For ease of reading, denote by $\tilde{x}$, $\tilde{y}$ and $\tilde{z}$ the images of $x$, $y$ and $z$ respectively under $\LI{2}$.
When $x$ and $y$ are orthogonal \eqref{eq:double} implies
\[
\GrMor{}{\bullet}{\tilde{x}}{\tilde{y}}=\GrMor{}{\bullet}{\tilde{y}}{\tilde{x}}=0,
\]
and so the associated twists commute.

To prove (\ref{key:braid}) note that the triangle $y\to z\to x\to y[1]$ induces a non-trivial triangle in $\catGQ{2}$ via $\LI{2}$.
By \cite[Lemma~4.2]{MR3281136}
\[
\GrMor{}{\bullet}{x}{y}\cong\mathbf{k}[-1]\quad\text{and}\quad\GrMor{}{\bullet}{y}{x}=0.
\]
Thus \eqref{eq:double} yields $\GrMor{}{\bullet}{\tilde{x}}{\tilde{y}} \cong \mathbf{k}[-1]$ and $\GrMor{\bullet}{\tilde{y}}{\tilde{x}}\cong\mathbf{k}[-1]$, and we deduce that $\tilde{z}=\varphi_{\tilde{x}}(\tilde{y})=\varphi_{\tilde{y}}^{-1}(\tilde{x})$.  Therefore
\[
  \varphi_{\tilde{x}}\circ\varphi_{\tilde{y}}\circ\varphi_{\tilde{x}}^{-1}
    =\varphi_{\tilde{z}}=\varphi_{\tilde{y}}^{-1}\circ\varphi_{\tilde{x}}\circ\varphi_{\tilde{y}},
    \]
as required.
\end{proof}

\begin{construction}\label{cons1}
We associate to any t-structure in $\tiltQ$ the generating set $\{b(t_1),\ldots,b(t_n)\}$ of $\Br(Q)$  where  $\{t_1,\ldots,t_n\}$ are the simple objects of the heart. The generating set associated to $\Qzero$ is the standard one.
\end{construction}
The following proposition gives an alternative inductive construction of these generating sets which we use in the sequel.
\begin{proposition}\label{pp:key}
Suppose $\ts{D}$ is a t-structure in $\tiltQo{N} \subseteq \tiltQ$. Then
\begin{enumerate}
\item[(i)] if $x$ and $y$ are two simple objects in the heart of $\cat{D}$  one has
    \[\begin{cases}
    b(x) b(y)= b(y) b(x), &\text{if $\GrMor{}{\bullet}{x}{y} = \GrMor{}{\bullet}{y}{x}= 0$},\\
    b(x) b(y) b(x)= b(y) b(x) b(y),
         &\text{otherwise}.
\end{cases}\]
\item[(ii)] if $\{t_i\}$ is the set of simple objects in the heart of $\cat{D}$,  the simple objects of the heart of $L_{t_i} \cat{D}$ are
    \begin{equation}
    \label{simple change 2}
        \{t_i[-1]\}\cup\{t_k\ \colon\  \GrMor{}{1}{t_i}{t_k}=0, k\neq i\}\cup\{\varphi_{t_i}(t_j)\ \colon\  \GrMor{}{1}{t_i}{t_j}\neq0\}
    \end{equation}
    and the corresponding associated generating set of $\Br(Q)$ is
    \begin{equation}
    \label{generator change 2}
        \{b_i\}\cup\{b_k\ \colon\  \GrMor{}{1}{t_i}{t_k}=0, k\neq i\}\cup\{b_i b_j b_i^{-1}\ \colon\  \GrMor{}{1}{t_i}{t_j}\neq0\},
    \end{equation}
    where $\{b_i:=b(t_i)\}$ is the generating set associated to $\cat{D}$.
\end{enumerate}
In particular, any such associated set is indeed a generating set of $\Br(Q)$. Here in (\ref{simple change 2}) we use the notation $\varphi_{a}(b):=\Cone\left(a\otimes\GrMor{}{\bullet}{a}{b}\to a\right)$ even when $a$ is not a spherical object.
\end{proposition}
\begin{proof}
First we note that \eqref{simple change 2} in (ii) is a special case of \cite[Proposition~5.4]{king-qiu}. The necessary conditions to apply this proposition follow from \cite[Theorem 5.9 and Proposition 6.4]{king-qiu}.

For (i), if $x$ and $y$ are mutually orthogonal then the commutative relations follow from (2) of Lemma~\ref{lem:key}.
Otherwise, by \cite[Lemma~4.2]{MR3281136},
\[
\GrMor{}{\bullet}{x}{y}\cong\mathbf{k}[-d]\quad\text{and}\quad\GrMor{}{\bullet}{y}{x}=0.
\]
for some strictly positive integer $d$. By \eqref{simple change 2}, after tilting $\cat{D}$
with respect to the simple object $x$ (and its shifts) $d$ times we reach a  heart with a simple object $z=\varphi_{x}(y)$.
In particular, there is a triangle $z\to x[-d]\to y\to z[1]$ in $\cat{D}(Q)$ where $z\in\Ind\cat{D}(Q)$. The braid relation then follows from (3) of Lemma~\ref{lem:key}.

Finally, \eqref{generator change 2} in (ii) follows from a direct calculation.
\end{proof}

We can use this construction to associate generating sets to t-structures in $\tiltGQo{N} \subseteq \tiltGQ{N}$.
Let $\ts{E}$ be such a t-structure, and $\{s_i\}$ the set of simple objects of its heart. Then  $(\LI{N})^{-1}s_i$ is well-defined, and we associate the generating set $\{ b_{s_i}:= b\left( (\LI{N})^{-1}s_i\right) \}$ of $\Br(Q)$ to $\ts{E}$.
\begin{remark}
This construction only works for $\ts{E} \in \tiltGQo{N}$ because the simple objects of the hearts of other t-structures need not be in the image of the Lagrangian immersion. This is the same reason that the isomorphism \eqref{eq:==} cannot be extended to the whole of $\tiltQ$.
\end{remark}
The next result follows immediately from Proposition~\ref{pp:key}.
\begin{corollary}\label{cor:key}
Let $\ts{E}\in \tiltGQo{N}$, and let $\{s_i\}$ be the set of simple objects in its heart, with corresponding generating set $\{b_{s_i}\}$. Then
\[\begin{cases}
    b_{s_i} b_{s_j}= b_{s_j} b_{s_i}, &\text{if $\GrMor{}{\bullet}{s_i}{s_j} = 0$},\\
    b_{s_i} b_{s_j} b_{s_i}= b_{s_j} b_{s_i} b_{s_j}, &\text{otherwise}.
\end{cases}\]
Moreover, the simple objects of the heart of $L_{s_i} \ts{E}$ are
    \begin{equation}
    \label{simple change 3}
        \{s_i[-1]\}\cup\{s_k\ \colon\  \GrMor{}{1}{s_i}{s_k}=0, k\neq i\}\cup\{\varphi_{s_i}(s_j)\ \colon\  \GrMor{}{1}{s_i}{s_j}\neq0\}
    \end{equation}
and the corresponding associated generating set is
    \begin{equation}
    \label{generator change 3}
        \{b_{s_i}\}\cup\{b_{s_k}\ \colon\  \GrMor{}{1}{s_i}{s_k}=0, k\neq i\}\cup\{b_{s_i} b_{s_j} b_{s_i}^{-1}\ \colon\  \GrMor{}{1}{s_i}{s_j}\neq0\}.
    \end{equation}
\end{corollary}

\begin{lemma}
\label{tilting dichotomy}
Let $s$ be a simple object in the heart of $\ts{E} \in \tiltGQo{N}$. Then either $L_s\ts{E} \in \tiltGQo{N}$ or $\varphi_s^{-1}L_s\ts{E} \in \tiltGQo{N}$. The first case occurs if and only if, in addition, $s\in \GQzero[3-N]$.
\end{lemma}
\begin{proof}
By \cite[Corollary 8.4]{king-qiu} the spherical twist $\varphi_s$ takes $\ts{E}$ to the t-structure obtained from it by tilting $N-1$ times `in the direction of $s$', i.e.\ by tilting at $s, s[-1], s[-2],\ldots, s[3-N]$ and finally $s[2-N]$. The first statement then follows from the isomorphism $\tiltQo{N} \cong \tiltGQo{N}$ of \cite[Theorem 8.1 and Proposition 5.13]{king-qiu}. For the second statement note that if $L_s\ts{E} \in \tiltGQo{N}$ then $s[-1] \in \GQzero[2-N]$, so $s\in \GQzero[3-N]$, and conversely if $s\not \in \GQzero[3-N]$ then $s[-1]\not \in \GQzero[2-N]$ which implies $L_s\ts{E} \not \in \tiltGQo{N}$.
\end{proof}
The above lemma justifies the following definition.
\begin{definition}
Let $\CoveringPoset$ be the poset whose underlying set is \[\braidQ\times \tiltGQo{N},\] and whose relation is generated by $(b,\ts{E})\leq (b',\ts{E}')$ if either $b=b'$ and $\ts{E} \leq \ts{E}'$ in $\tiltGQo{N}$, or $b'=b\cdot b_s$ and $\ts{E}' = \varphi_s^{-1}L_s\ts{E}$ where $s$ is a simple object of the heart of $\ts{E}$ with the property that $L_s\ts{E} \not \in \tiltGQo{N}$, equivalently, by Lemma~\ref{tilting dichotomy}, $s\not \in \GQzero[3-N]$.
\end{definition}
\begin{lemma}
\label{P maps to tilt}
There is a map of posets \[\alpha \colon \CoveringPoset \to \tiltGQ{N} \colon (b,\ts{E}) \mapsto b\cdot \ts{E}:=\braidsurj_N(b)\ts{E},\] which is surjective on objects and on morphisms. Moreover, $\CoveringPoset$ is connected and $\alpha$ is equivariant with respect to the canonical free left $\braidQ$-action on $\CoveringPoset$.
\end{lemma}
\begin{proof}
To check that $\alpha$ is a map of posets we need only check that the generating relations for $\CoveringPoset$ map to relations in $\tiltGQ{N}$. This is clear since (in either case) $b'\cdot \ts{E}' = b\cdot L_s\ts{E} = L_{b \cdot s} \left( b\cdot \ts{E} \right)$. It is surjective on objects by \cite[Proposition 8.3]{king-qiu}. To see that it is surjective on morphisms it suffices to check that each morphism $\ts{F} \leq L_t\ts{F}$, where $t$ is a simple object of the heart of $\ts{F}$, lifts to $\CoveringPoset$. For this, suppose $\ts{F} = b \cdot \ts{E}$ where $\ts{E} \in \tiltGQo{N}$, and that $t = b\cdot s$ for simple $s$ in the heart of $\ts{E}$. Then either $L_s \ts{E} \in \tiltGQo{N}$ and $(b,\ts{E}) \leq (b,L_s\ts{E})$ is the required lift, or $L_s \ts{E} \not \in \tiltGQo{N}$ and
\[
(b,\ts{E}) \leq (b\cdot b_s,\varphi_s^{-1}L_s\ts{E})
\]
is the required lift.

The connectivity of $\CoveringPoset$ follows from the facts that $(b,\ts{E}) \leq (b\cdot b_s, \ts{E})$ for any simple object $s$ of the heart of $\ts{E} \in \tiltGQo{N}$ and that $\tiltGQo{N}$ is connected. Finally, the equivariance with respect to the left $\braidQ$-action $b'\cdot (b,\ts{E}) = (b'b,\ts{E})$ is clear.
\end{proof}

\begin{proposition}
\label{P covers tilt}
The morphism $\alpha \colon \CoveringPoset \to \tiltGQ{N}$ is a covering.
\end{proposition}
\begin{proof}
By Lemma~\ref{P maps to tilt} we know $\alpha$ is surjective on objects and on morphisms, so all we need to show is that each morphism lifts {\em uniquely} to $\CoveringPoset$ once the source is given. By Remark~\ref{generators and relations for tilt} it suffices to show that the squares and pentagons (\ref{45}) of Lemma~\ref{lem:45} lift to $\CoveringPoset$. Using the $\braidQ$-action on $\CoveringPoset$ it suffices to show that the diagrams with source $\ts{D}$ lift to diagrams with source $(1,\ts{D})$. We treat only the case of the pentagon, since the square is similar but simpler. We use the notation of Lemma~\ref{lem:45}: $s_i$ and $s_j$ are simple objects in the heart of $\ts{D} \in \tiltGQo{N}$ with $\GrMor{}{1}{s_i}{s_j} \cong k$ and $\GrMor{}{1}{s_j}{s_i} \cong 0$, and $e$ is the extension sitting in the non-trivial triangle $s_j \to e \to s_i \to s_j[1]$.

There are four cases depending on whether or not $L_{s_i}\ts{D}$ and $L_{s_j}\ts{D}$ are in $\tiltGQo{N}$ or not.
\begin{description}
\item[Case A] If $L_{s_i}\ts{D}, L_{s_j}\ts{D} \in \tiltGQo{N}$ then $L_{\langle s_i,s_j\rangle}\ts{D} = L_{s_i}\ts{D} \vee L_{s_j}\ts{D} \in \tiltGQo{N}$ too. Hence there is obviously a lifted diagram in $1\times \tiltGQo{N}$.
\item[Case B] If $L_{s_i}\ts{D} \not \in  \tiltGQo{N}$ but $L_{s_j}\ts{D} \in \tiltGQo{N}$ then we claim that
\[
\xymatrix@C=1pc@R=0.8pc{
    &   (b_{s_i}, \varphi_{s_i}^{-1}L_{s_i}\ts{D})  \ar[rr]^{\varphi_{s_i}^{-1}e} &&      (b_{s_i}, \varphi_{s_i}^{-1}\ts{D}' \ar[dd]^{\varphi_{s_i}^{-1}s_j}) \\
  (1, \ts{D}) \ar[ur]^{s_i\quad}\ar[dr]_{s_j\quad}     \\
    &   (1,L_{s_j}\ts{D})   \ar[rr]_{s_i\qquad}  &&(b_{s_i} , \varphi_{s_i}^{-1} L_{\langle s_i, s_j\rangle}\ts{D})
    }
\]
is the required lift. (Here, and in the sequel, we label the morphisms by the associated simple object.) To confirm this we note that by Lemma~\ref{tilting dichotomy} $s_i \not \in \GQzero[3-N]$, from which it follows that the bottom morphism is in $\CoveringPoset$, and that similarly $\varphi_{s_i}^{-1}e=s_j \in \GQzero[3-N]$ so that the top morphism is in $\CoveringPoset$. It follows that the right hand morphism is in $\CoveringPoset$ too, because $\varphi_{s_i}^{-1} L_{\langle s_i, s_j\rangle}\ts{D} \in \tiltGQo{N}$.
\item[Case C] If $L_{s_i}\ts{D} \in  \tiltGQo{N}$ but $L_{s_j}\ts{D} \not \in \tiltGQo{N}$ then one can verify that
\[
\xymatrix@C=1pc@R=0.8pc{
    &   (1, L_{s_i}\ts{D})  \ar[rr]^{e} &&      (1, \ts{D}') \ar[dd]^{s_j} \\
  (1, \ts{D}) \ar[ur]^{s_i\quad}\ar[dr]_{s_j\quad}     \\
    &   (b_{s_j},\varphi_{s_j}^{-1}L_{s_j}\ts{D})   \ar[rr]_{\varphi_{s_j}^{-1}s_i\quad}  &&(b_{s_j} , \varphi_{s_j}^{-1} L_{\langle s_i, s_j\rangle}\ts{D})
    }
\]
is the required lift when $\varphi_{s_j}^{-1}s_i = e \in \GQzero[3-N]$. If $e \not \in \GQzero[3-N]$ then
\[
\xymatrix@C=1pc@R=0.8pc{
    &   (1, L_{s_i}\ts{D})  \ar[rr]^{e} &&      (b_e, \varphi_e^{-1}\ts{D}') \ar[dd]^{\varphi_e^{-1} s_j} \\
  (1, \ts{D}) \ar[ur]^{s_i\quad}\ar[dr]_{s_j\quad}     \\
    &   (b_{s_j},\varphi_{s_j}^{-1}L_{s_j}\ts{D})   \ar[rr]_{\varphi_{s_j}^{-1}s_i\qquad}  &&(b_{s_j}b_e , \varphi_e^{-1}\varphi_{s_j}^{-1} L_{\langle s_i, s_j\rangle}\ts{D})
    }
\]
is the required lift. We need only check that the right-hand morphism is in $\CoveringPoset$. For this note that $\varphi_e^{-1}s_j = s_i[-1]$ so that $b_{\varphi_e^{-1}s_j} = b_{s_i}$, and that applying  (\ref{key:braid}) of Lemma~\ref{lem:key} to the triangle $s_i[-1] \to s_j \to e \to s_i$ we have $b_{s_j} = b_e b_{s_i} b_e^{-1}$, or equivalently $b_{s_j}b_e = b_eb_{\varphi_e^{-1}s_j}$. Moreover, since
\[
\varphi^{-1}_{\varphi_e^{-1}s_j} L_{\varphi_e^{-1}s_j}\varphi_e^{-1}\ts{D}' =
\varphi^{-1}_{\varphi_e^{-1}s_j} \varphi_e^{-1} L_{s_j}\ts{D}' =
\varphi_e^{-1}\varphi^{-1}_{s_j}  L_{\langle s_i,s_j\rangle }\ts{D},
\]
and we already know the latter is in $\tiltGQo{N}$, we see that the right-hand morphism is indeed in $\CoveringPoset$.
\item[Case D] If $L_{s_i}\ts{D}, L_{s_j}\ts{D} \not \in \tiltGQo{N}$ then the lifted pentagon is
\[
\xymatrix@C=1pc@R=0.8pc{
    &   (b_{s_i}, \varphi_{s_i}^{-1}L_{s_i}\ts{D})  \ar[rr]^{\varphi_{s_i}^{-1}e} &&      (b_{s_i}b_{s_j}, \varphi_{s_j}^{-1}\varphi_{s_i}^{-1}\ts{D}' \ar[dd]^{\varphi_{s_j}^{-1}\varphi_{s_i}^{-1}s_j}) \\
  (1, \ts{D}) \ar[ur]^{s_i\quad}\ar[dr]_{s_j\quad}     \\
    &   (b_{s_j},\varphi_{s_j}^{-1}L_{s_j}\ts{D})   \ar[rr]_{\varphi_{s_j}^{-1}s_i\qquad}  &&(b_{s_j}b_e , \varphi_e^{-1}\varphi_{s_j}^{-1} L_{\langle s_i, s_j\rangle}\ts{D})
    }
\]
The top morphism is in $\CoveringPoset$ because $\varphi_{s_i}^{-1}e = s_j \not \in \GQzero[3-N]$. The bottom morphism is in $\CoveringPoset$ because $\varphi_{s_j}^{-1}s_i=e \not \in \GQzero[3-N]$, for if it were then $s_i$ would be in $\GQzero[3-N]$, which is false by assumption. It remains to check that the right-hand morphism is in $\CoveringPoset$. Note that
\[
L_{\varphi_{s_j}^{-1}\varphi_{s_i}^{-1}s_j} \varphi_{s_j}^{-1}\varphi_{s_i}^{-1}\ts{D}' =
\varphi_{s_j}^{-1}\varphi_{s_i}^{-1} L_{s_j} \ts{D}' =
\varphi_{s_j}^{-1}\varphi_{s_i}^{-1} L_{\langle s_i,s_j\rangle} \ts{D}.
\]
Therefore, since we already know that $\varphi_e^{-1}\varphi_{s_j}^{-1} L_{\langle s_i, s_j\rangle}\ts{D} \in \tiltGQo{N}$, it suffices to show that $b_{s_i}b_{s_j} = b_{s_j}b_e$, since it then follows that $\varphi_{s_j}^{-1}\varphi_{s_i}^{-1} = \varphi_e^{-1}\varphi_{s_j}^{-1}$. The required equation is obtained by applying (\ref{key:braid}) of Lemma~\ref{lem:key} to the triangle $e\to s_i \to s_j[1] \to e[1]$, and recalling that $b$ is invariant under shifts.\qedhere
\end{description}
\end{proof}

\begin{corollary}
\label{braid action free 1}
For $N\geq 2$, the map $\alpha \colon \CoveringPoset \to \tiltGQ{N}$ is a $\braidQ$-equivariant isomorphism, and in particular $\braidQ$ acts freely on $\tiltGQ{N}$. The map $\braidsurj_N \colon \braidQ \to \braidGQ{N}$ is an isomorphism.
\end{corollary}
\begin{proof}
This follows immediately from the fact that $\tiltGQ{N}$ is contractible, i.e.\ has contractible classifying space, and that $\alpha \colon \CoveringPoset \to \tiltGQ{N}$ is a connected $\braidQ$-equivariant cover on which $\braidQ$ acts freely.

Recall that $\braidQ$ acts on $\tiltGQ{N}$ via the surjective homomorphism $\braidsurj_N$. Since the action is free $\braidsurj_N$ must also be injective, and therefore is an isomorphism.
\end{proof}
\begin{remark}
When $Q$ is of type A, Corollary~\ref{braid action free 1} provides a third proof of Theorem~\ref{thm:ST}.
When $Q$ is of type E, it shows that there is a faithful symplectic representation of the braid group,
because $\catGQ{N}$ is a subcategory of a derived Fukaya category, while the spherical twists are the higher version of Dehn twists. This is contrary to the result in \cite{MR1719815} in the surface case, which says that there is
no faithful geometric representation of the braid group of type $E$.
\end{remark}

\begin{corollary}
\label{braid action free 2}
For $N\geq 2$, the induced action of $\braidQ$ on $\stabGQ{N}$ is free.
\end{corollary}
\begin{proof}
If an element of $\braidQ$ fixes $\sigma \in \stabGQ{N}$ then it must fix the associated t-structure in $\tiltGQ{N}$.
\end{proof}
Note that we recover the well-known fact that $\braidQ$ is torsion-free from this last corollary because $\stabGQ{N}$ is contractible and $\braidQ$ acts freely so $\stabGQ{N}/\braidQ$ is a {\em finite-dimensional} classifying space for $\braidQ$. The classifying space of any group with torsion must be infinite-dimensional.

\subsection{Higher cluster theory}
\label{higher cluster theory}

The quotient $\tiltGQ{N} / \braidQ$ has a natural description in terms of higher cluster theory. We recall the relevant notions from \cite[Secion~4]{king-qiu}. As previously, $\cat{D}(Q)$ is the bounded derived category of the quiver $Q$.
\begin{definition}
\label{def:cluster}
For any integer $m\geq 2$, the \emph{$m$-cluster shift} is the auto-equivalence of $\cat{D}(Q)$ given by
$\shift{m}=\tau^{-1}\circ[m-1]$, where $\tau$ is the Auslander--Reiten translation. The \defn{$m$-cluster category} $\cluster{m}{Q}= \cat{D}(Q)/\shift{m}$ is the orbit category,
which is Calabi--Yau-$m$. When it is clear from the context we will omit the index $m$ from the notation.
\end{definition}

 An \defn{$m$-cluster tilting set} $\{p_j\}_{j=1}^n$ in $\cluster{m}{Q}$ is an Ext-configuration, i.e.\ a maximal collection of non-isomorphic indecomposable objects such that
 \[
 \Ext^k_{\cluster{m}{Q}}(p_i, p_j)=0,\ \text{for}\ 1\leq k\leq m-1.
 \]
Any $m$-cluster tilting set consists of $n=\rk K\cat{D}(Q)$ objects.

   New cluster tilting sets can be obtained by mutations. The \defn{forward mutation} $\mu_{p_i}^\sharp P$ of an $m$-cluster tilting set $P=\{p_j\}_{j=1}^n$ at the object $p_i$ is obtained by replacing $p_i$ by
    \[
        p_i^\sharp = \Cone(p_i \to \bigoplus_{j\neq i} \Irr(p_i,p_j)^*\otimes p_j).
    \]
    Here $\Irr(p_i,p_j)$ is the space of irreducible maps from $p_i$ to $p_j$ in the full additive subcategory $\Add \left(\bigoplus_{i=1}^n p_i\right)$ of $\cluster{m}{Q}$ generated by the objects of the original cluster tilting set. Similarly, the \emph{backward mutation} $\mu_{p_i}^\flat P$ is obtained by replacing $p_i$ by
    \[
      p_i^\flat =  \Cone(\bigoplus_{j\neq i} \Irr(p_j,p_i)\otimes p_j \to p_i)[-1].
    \]
    As the names suggest, forward and backward mutation are inverse processes.

    Cluster tilting sets in $\cluster{N-1}{Q}$ and their mutations are closely related to t-structures in $\catGQ{N}$ and tilting between them. To be more precise,  \cite[Theorem~8.6]{king-qiu}, based on the construction of \cite[\S2]{MR2640929}, states that $(N-1)$-cluster tilting sets are in bijection with the $\braidQ$-orbits in $\tiltGQ{N}$, and that a cluster tilting set $P'$ is obtained from $P$ by a backward mutation if and only if each t-structure in the orbit corresponding to $P'$ is obtained by a simple left tilt from one in the orbit corresponding to $P$. This motivates the following definition.
\begin{definition}
\label{cluster mutation category}
The \defn{cluster mutation category} $\mutation{N-1}{Q}$ is the category whose objects are the $(N-1)$-cluster tilting sets, and whose morphisms are generated by backward mutations subject to the relations that for distinct $p_i,p_j\in P$ the diagrams
\begin{gather}\label{45v2}
\xymatrix@C=1.4pc@R=1pc{
    &  \mu^\flat_{p_i}P    \ar[dr]^{}\\
    P  \ar[ur]^{} \ar[dr]_{} &&  \mu^\flat_{p_j}\mu^\flat_{p_i}P \\
    &   \mu^\flat_{p_j}P   \ar[ur]_{}
}\qquad
\xymatrix@C=1pc@R=0.8pc{
    &   \mu^\flat_{p_i}P   \ar[rr]^{} &&      \mu^\flat_{p_i}\mu^\flat_{p_j}P \ar[dd]^{} \\
   P \ar[ur]^{}\ar[dr]_{}     \\
    &   \mu^\flat_{p_j}P   \ar[rr]_{}  &&\mu^\flat_{p_j}\mu^\flat_{p_i}P
}
\end{gather}
commute whenever there is a corresponding lifted diagram of simple left tilts in $\tiltGQ{N}$. Note that, possibly after switching the indices $i$ and $j$ in the pentagonal case, there is always a diagram of one of the above two types.
\end{definition}

\begin{proposition}
\label{mutations from tilts}
There is an isomorphism of categories
\[
\tiltGQ{N}/ \braidQ \cong \mutation{N-1}{Q}.
\]
The classifying space of $\mutation{N-1}{Q}$ is a $K(\braidQ,1)$.
\end{proposition}
\begin{proof}
The first statement is a rephrasing of \cite[Theorem~8.6]{king-qiu}, using Remark~\ref{generators and relations for tilt} and the definition of $\mutation{N-1}{Q}$. The second statement follows from the first and the fact that $\tiltGQ{N}$ is contractible, and the $\braidQ$-action on it free.
\end{proof}

\subsection{Garside groupoid structures}
\label{garside}

In \cite[\S1]{MR2395163} a Garside groupoid is defined as a group $G$ acting freely on the left of a lattice $L$ in such a way that
\begin{itemize}
\item the orbit set $G \backslash L$ is finite;
\item there is an automorphism $\psi$ of $L$ which commutes with the $G$-action;
\item for any $l\in L$ the interval $[l,l\psi]$ is finite;
\item the relation on $L$ is generated by $l\leq l'$ whenever $l'\in [l,l\psi]$.
\end{itemize}
The action of $\braidQ$ on $\tiltGQ{N}$ provides an example for any $N\geq 3$, in fact a whole family of examples. By Corollary~\ref{braid action free 1} the action is free, and by (\ref{eq:fund}) the orbit set is finite. From \S~\ref{ftt cpts} we know that $\tiltGQ{N}$ is a lattice, and that closed bounded intervals within it are finite. It remains to specify an automorphism $\psi$; we choose $\psi= [-d]$ for any integer $d\geq 1$. It is then clear that the last condition is satisfied since each simple left tilt from $\ts{D}$ is in the interval between $\ts{D}$ and $\ts{D}[-d]$.

In fact the preferred definition of Garside groupoid in \cite{MR2395163} is that given in \S3, not \S 1, of that paper. There a Garside groupoid $\cat{G}$ is defined to be the groupoid associated to a category $\cat{G}^+$ with a special type of presentation ---  called a complemented presentation ---  together with an automorphism $\varphi\colon \cat{G} \to \cat{G}$ (arising from an automorphism of the presentation) and a natural transformation $\Delta \colon 1 \to \varphi$ such that
\begin{itemize}
\item the category $\cat{G}^+$ is atomic, i.e.\ for each morphism $\gamma$ there is some $k\in \N$ such that $\gamma$ cannot be written as a product of more than $k$ non-identity morphisms;
\item the presentation of $\cat{G}$ satisfies the cube condition, see \cite[\S3]{MR2395163} for the definition;
\item for each $g\in \cat{G}^+$ the natural morphism $\Delta_g \colon g \to \varphi(g)$ factorises through each generator with source $g$.
\end{itemize}
The naturality of $\Delta$ is equivalent to the statement that for any generator $\gamma \colon g \to g'$ we have $\Delta_{g'} \circ \gamma = \varphi(\gamma) \circ \Delta_g$. The collection of data of a complemented presentation, an automorphism, and a natural transformation satisfying the above properties is called a \defn{Garside tuple}. See \cite[Theorem~3.2]{MR2395163} for a list of the good properties of a Garside tuple.

Briefly, the translation from the second to the first form of the definition is as follows. Fix an object $g\in \cat{G}^+$. Let $L = \Mor{\cat{G}}{g}{-}$ with the order $\gamma \leq \gamma' \iff \gamma^{-1}\gamma' \in \cat{G}^+$. Let $G = \Mor{\cat{G}}{g}{g}$ acting on $L$ via pre-composition. Let the automorphism $\psi$ be given by taking $\gamma \colon g\to g'$ to $\varphi(\gamma) \circ \Delta_g \colon g \to \varphi(g) \to \varphi(g')$. Note that with these definitions the interval $[\gamma, \gamma\psi]$ in the lattice consists of the initial factors of the morphism $\Delta_{g'}$ in the category $\cat{G}^+$.

Below, we verify that cluster mutation category $\mutation{N-1}{Q}$ forms part of a Garside tuple.
\begin{proposition}
\label{garside tuple}
Let the category $\cat{G}^+$ be $\mutation{N-1}{Q}$, where $N\geq 2$, presented as in Definition~\ref{cluster mutation category}. Let the automorphism $\varphi = [-d]$ for an integer $d\geq 1$. Let the natural transformation $\Delta_{P} \colon P  \to P[-d]$ be given by the image under the isomorphism $\tiltGQ{N}/ \braidQ \cong \mutation{N-1}{Q}$ of the unique morphism in $\tiltGQ{N}$ from an object to its shift by $[-d]$. Then $(\cat{G}^+, \varphi, \Delta)$ is a Garside tuple.
\end{proposition}
\begin{proof}
It is easy to check that the presentation in Definition~\ref{cluster mutation category} is complemented ---  see \cite[\S3]{MR2395163} for the definition. The atomicity of $\mutation{N-1}{Q}$ follows from the fact that closed bounded intervals in the cover $\tiltGQ{N}$ are finite, since this implies that any morphism has only finitely many factorisations into non-identity morphisms. The factorisation property follows from the inequalities
\[
\ts{D} \leq L_s\ts{D} \leq \ts{D}[-d]
\]
for any simple object $s$ of the heart of any t-structure $\ts{D}$. Finally the cube condition follows from the fact that the cover $\tiltGQ{N}$ is a lattice.
\end{proof}
\begin{remark}
In the case $N=3$ and $d=1$ the natural morphism $\Delta_P$ is a maximal green mutation sequence, in the sense of Keller (cf. \cite{MR2931896} and \cite{Q}). For  $N> 3$ and $d=N-2$, the natural transformation $\Delta$ should be thought as the generalised, or higher, green mutation (for Buan--Thomas's coloured quivers, cf. \cite[\S6]{king-qiu}).
\end{remark}

Finally we explain the relationship of the above Garside structure to that on the braid group $\braidQ$ as described in, for example, \cite{MR2854121}. Suppose the automorphism $\varphi$ fixes some object $g\in \cat{G}$. Let $G=\Mor{\cat{G}}{g}{g}$, and define the monoid $G^+$ analogously. Then we claim $G^+$ is a Garside monoid, and $G$ the associated Garside group ---  the properties of a complemented presentation ensure that $G^+$ is finitely generated by those generators of $\cat{G}^+$ with source and target $g$, and also that it is a cancellative monoid; moreover $G^+$ is atomic since $\cat{G}^+$  is;  the cube condition ensures that the partial order relation defined by divisibility in $G^+$ is a lattice; and finally the natural transformation $\Delta$ yields a central element $\Delta_g\in Z(G)$, which plays the r\^{o}le of Garside element.

As a particular example note that the automorphism $\varphi = [ k(2-N)]$, where $k\in \N$, fixes the standard cluster tilting set in $\mutation{N-1}{Q}$. By Proposition~\ref{mutations from tilts} the group of automorphisms is $\braidQ$, and thus we obtain a Garside group structure on $\braidQ$. For a suitable choice of $k$ this agrees with that described in \cite{MR2854121}.


\end{document}